\documentclass[twoside,11pt,reqno]{amsart}
\usepackage{amsmath,amssymb,amscd,mathrsfs,epic,wasysym,latexsym,tikz,mathrsfs,cite,hyperref}
\usepackage{pb-diagram}
\usepackage{tikz-cd}

\makeatletter

\numberwithin{equation}{section}

\allowdisplaybreaks

\newtheorem{Proposition}[equation]{Proposition}
\newtheorem{Lemma}[equation]{Lemma}
\newtheorem{Theorem}[equation]{Theorem}
\newtheorem{Corollary}[equation]{Corollary}
\newtheorem{MainTheorem}{Theorem}

\theoremstyle{definition}  

\newtheorem{Remark}[equation]{Remark}

\newtheorem{Example}[equation]{Example}

\let\<\langle
\let\>\rangle

\newcommand\Comment[2][\relax]{\space\par\medskip\noindent%
   \fbox{\begin{minipage}{\textwidth}\textbf{Comment\ifx\relax#1\else---#1\fi}\newline%
        #2\end{minipage}}\medskip
}

\def\bs{\text{\boldmath$s$}}
\def\bt{\text{\boldmath$t$}}
\def\bc{\text{\boldmath$c$}}
\def\br{\text{\boldmath$r$}}

\def\b1{\text{\boldmath$1$}}

\def\bM{\text{\boldmath$M$}}

\def\bd{\text{\boldmath$d$}}

\def\alt{{\tt alt}}

\def\Tens{\operatorname{{\mathtt{Tens}}}}

\newcommand{\ttbb}{\mathrm{\mathbf b}}

\def\ula{\text{\boldmath$\lambda$}}
\def\umu{\text{\boldmath$\mu$}}
\def\unu{\text{\boldmath$\nu$}}

\def\pmod#1{\text{ }(\text{\rm mod } #1)\,}

\newcommand{\Hom}{\operatorname{Hom}}

\newcommand{\End}{\operatorname{End}}

\newcommand{\Z}{\mathbb{Z}}

\newcommand{\0}{{\bar 0}}
\newcommand{\1}{{\bar 1}}
\def\eps{{\varepsilon}}
\def\phi{{\varphi}}

\newcommand{\zc}{{\textsf{c}}}

\newcommand{\ze}{{\textsf{e}}}

\newcommand{\za}{{\textsf{a}}}

\newcommand{\F}{{\mathbb F}}

\newcommand{\Ga}{\Gamma}
\newcommand{\la}{\lambda}
\newcommand{\La}{\Lambda}
\newcommand{\al}{\alpha}
\newcommand{\be}{\beta}

\def\Si{\mathfrak{S}}
\newcommand{\si}{\sigma}
\newcommand{\om}{\omega}
\newcommand{\Om}{\Omega}

\newcommand{\de}{\delta}
\newcommand{\De}{\Delta}

\def\triv#1{\O_{#1}}

\newcommand{\Ind}{{\mathrm {Ind}}}

\newcommand{\Mat}{{\mathcal {M}}}
\def\rank{\mathop{\mathrm{ rank}}\nolimits}

\newcommand{\ad}{{\operatorname{ad}}}
\newcommand{\Sym}{\operatorname{{\mathtt{Sym}}}}
\newcommand{\Inv}{\operatorname{{\mathtt{Inv}}}}

\newcommand{\Q}{{\mathbb Q}}

\newcommand{\D}{{\mathscr D}}

\newcommand{\DC}{\mathcal D}
\newcommand{\GC}{{\mathcal G}}

\newcommand{\m}{\mathfrak m}

\newcommand{\Zig}{{\sf Z}}

\def\b{\mathfrak{b}}
\def\k{\Bbbk}
\def\K{\mathbb K}

\def\into{{\hookrightarrow}}

\renewcommand\O{\mathcal O}
\newcommand\OO{{}'{\mathcal M}}

\def\iso{\stackrel{\sim}{\longrightarrow}}

\def\ttB{{\mathtt B}}
\def\ttb{{\mathtt b}}

\def\triv{{\tt triv}}
\def\col{{\tt col}}
\def\row{{\tt row}}

\def\lan{\langle}
\def\ran{\rangle}

\def\Seq{{\tt Seq}}

\newcommand{\bC}{\mathbf{C}}
\newcommand{\bD}{\mathbf{D}}
\newcommand{\bE}{\mathbf{E}}
\newcommand{\bF}{\mathbf{F}}
\newcommand{\bG}{\mathbf{G}}

\newcommand{\Star}{\operatorname{\mathtt{Star}}}

{\catcode`\|=\active
  \gdef\set#1{\mathinner{\lbrace\,{\mathcode`\|"8000%
  \let|\midvert #1}\,\rbrace}}
}
\def\midvert{\egroup\mid\bgroup}

\colorlet{darkgreen}{green!50!black}
\tikzset{dots/.style={very thick,loosely dotted},
         greendot/.style={fill,circle,color=darkgreen,inner sep=1.5pt,outer sep=0},
         blackdot/.style={fill,circle,color=black,inner sep=1.5pt,outer sep=0},
         graydot/.style={fill,circle,color=gray,inner sep=1.1pt,outer sep=0}
}
\def\greendot(#1,#2){\node[greendot] at(#1,#2){}}
\def\blackdot(#1,#2){\node[blackdot] at(#1,#2){}}
\def\graydot(#1,#2){\node[graydot] at(#1,#2){}}

\newenvironment{braid}{
  \begin{tikzpicture}[baseline=6mm,black,line width=1pt, scale=0.32,
                      draw/.append style={rounded corners},
                      every node/.append style={font=\fontsize{5}{5}\selectfont}]%
  }{\end{tikzpicture}
}

\def\Grid(#1,#2){
  \draw[very thin,gray,step=2mm] (0,0)grid(#1,#2);
  \draw[very thin,darkgreen,step=10mm] (0,0)grid(#1,#2);
}

\newcommand\Tableau[2][\relax]{
  \begin{tikzpicture}[scale=0.5,draw/.append style={thick,black}]
    \ifx\relax#1\relax%
    \else 
      \foreach\box in {#1} { \filldraw[blue!30]\box+(-.5,-.5)rectangle++(.5,.5); }
    \fi
    \newcount\row\newcount\col
    \row=0
    \foreach \Row in {#2} {
       \col=1
       \foreach\k in \Row {
          \draw(\the\col,\the\row)+(-.5,-.5)rectangle++(.5,.5);
          \draw(\the\col,\the\row)node{\k};
          \global\advance\col by 1
       }
       \global\advance\row by -1
    }
  \end{tikzpicture}
}

\newcommand\YoungDiagram[2][\relax]{
  \begin{tikzpicture}[scale=0.5,draw/.append style={thick,black}]
    \ifx\relax#1\relax%
    \else 
    \foreach\box in {#1} {
      \filldraw[blue!30]\box rectangle ++(1,1);
    }
    \fi
    \newcount\row
    \row=0
    \foreach \col in {#2} {
       \draw(1,\the\row)grid ++(\col,1);
       \global\advance\row by -1
    }
  \end{tikzpicture}
}

\begin{document}

\title[Turner doubles and generalized Schur algebras]{{\bf Turner  doubles and generalized Schur algebras}}

\author{\sc Anton Evseev}
\address{School of Mathematics, University of Birmingham, Edgbaston, Birmingham B15 2TT, UK}
\email{a.evseev@bham.ac.uk}

\author{\sc Alexander Kleshchev}
\address{Department of Mathematics\\ University of Oregon\\Eugene\\ OR 97403, USA}
\email{klesh@uoregon.edu}


\thanks{The first author is supported by the EPSRC grant EP/L027283/1 and thanks the Max-Planck-Institut for hospitality. The second author is supported by the NSF grant DMS-1161094, the Max-Planck-Institut and the Fulbright Foundation.}

\begin{abstract}
Turner's Conjecture describes all blocks of symmetric groups and Hecke algebras up to derived  equivalence in terms of certain  double algebras. With a view towards a proof of this conjecture, we develop a general theory of Turner doubles.  In particular, we describe doubles as explicit maximal  symmetric subalgebras of certain generalized Schur algebras and establish a Schur-Weyl duality with wreath product algebras. 
\end{abstract}

\maketitle

\section{Introduction}
Turner's Conjecture \cite[Conjecture 165]{Turner} describes all blocks of symmetric groups and Hecke algebras up to derived  equivalence in terms of certain explicitly constructed {\em double algebras} $D_Q(n,d)$, where $Q$ is a quiver of finite type $A$. This paper is the first in a series of two papers where we prove Turner's Conjecture. To achieve this goal, in this paper  we develop a general theory of  Turner doubles, which we believe is of independent interest. 

For simplicity, in this introduction we describe the results only over the ground ring $\Z$. We fix a $\Z$-superalgebra $X=X_\0\oplus X_\1$ which is free of finite rank over $\Z$. Consider the invariants $\Inv^d X:=(X^{\otimes d})^{\Si_d}$ under the action of the symmetric group $\Si_d$. This action depends crucially on the superstructure on $X$, as do the structure and the dimension of $\Inv^d X$ and  of all algebras defined later in terms of $X$. 
There is a   natural superbialgebra structure on $\Inv X:=\bigoplus_{d\geq 0}\Inv^d X$. 
The {\em Turner double} is the superalgebra $D X:=\Inv X\otimes (\Inv X)^*$ with product defined in terms of the superbialgebra structures on $\Inv X$ and $(\Inv X)^*$. 

More precisely,  $(\Inv X)^*$ is naturally a superbimodule over $\Inv X$, and the product on $DX$ is described, using Sweedler's notation, as follows:
$$
(\xi\otimes x)(\eta\otimes y)=\sum \pm\, 
 \xi_{(2)}\eta_{(1)}\otimes (x\cdot\eta_{(2)})(\xi_{(1)}\cdot y),
$$
for homogeneous
$\xi,\eta\in \Inv X$ and $x,y\in (\Inv X)^*$, with signs   determined by superalgebra data. We explain in \S\ref{SSDoubles} why this agrees with Turner's definition in \cite{TurnerCat}. A key property of $D X$ is that it is always a {\em symmetric algebra}. Moreover, under some reasonable assumptions on $X$, the double $D X$ as well as all other algebras defined later in terms of $X$ are {\em non-negatively graded}.  In this case, the theorems below respect the gradings.

The superalgebra $(\Inv X)^*$ can be identified with the  symmetric superalgebra $\Sym(X^*)$, which is naturally a sublattice in the {\em divided power} superalgebra ${}'{\Sym(X^*)}$. We show that the superalgebra structure on $DX$ extends to that on ${}'{D X}:=\Inv X\otimes\, {}'{\Sym(X^*)}$. Thus $DX\subseteq {}'{D X}$ is a subsuperalgebra. Upon extension of scalars to a filed $\K$ of characteristic $0$, the embedding $DX\subseteq {}'{D X}$ induces an isomorphism $DX_\K\iso {}'{D X}_\K$. But, importantly, if  $\K$ has positive characteristic, the induced map is  neither injective nor surjective.

Let  $T_X=X\oplus X^*$ be the {\em trivial extension superalgebra} of $X$, with the product defined by 
$(\xi,x)(\eta,y)=(\xi\eta,\xi\cdot y+x\cdot \eta)$ for $\xi,\eta\in  X$ and $x,y\in X^*$. Let $*$ denote the shuffle product on $\bigoplus_{d\geq 0} (T_X)^{\otimes d}$. We show in Lemma~\ref{LemmaTri} that there is a natural isomorphism 
$
\kappa\colon {}'{\Sym(X^*)}\iso \Inv(X^*).
$
Our first main result is the following theorem, which often allows one to reduce the study of the double over $X$ to that of the invariants over $T_X$. 

\begin{MainTheorem}\label{TA}
We have:
\begin{enumerate}
\item[{\rm (i)}] The map  $\phi\colon {}'{D X}\to \Inv T_X, \ \xi\otimes x\mapsto \xi*\kappa(x)$ is an isomorphism of superalgebras. 
\item[{\rm (ii)}] The subalgebra $\phi(D X)\subseteq \Inv T_X$ is generated by $\Inv (X_\0)$ and all elements of the form $t* 1_X^{\otimes d}$ with $t\in T_X$ and $d\geq 0$.
\end{enumerate}
\end{MainTheorem}

We have a natural superalgebra decomposition $D X=\bigoplus_{d\geq 0} D^d X$, with $$D^d X=\bigoplus_{0\leq e\leq d} \Inv^e X \otimes (\Inv^{d-e} X)^*,$$ 
where the last direct sum is that of $\Z$-modules,  
and similarly for ${}'{D X}$. Then the isomorphism $\phi$ of Theorem~\ref{TA} restricts to isomorphisms $\phi\colon {}'{D^d X}\iso\Inv^d T_X$.  

Let $A$ be a $\Z$-superalgebra which is free of finite rank over $\Z$, and consider the case where $X$ is the matrix superalgebra $M_n(A)$ for some fixed $n$. 
In this case we use the special notation 
$$
 D^A(n,d):=D^d M_n(A),\quad {}'{D^A(n,d)}:={}'D^d M_n(A). 
$$
We refer to the superalgebra $D^A (n,d)$ as a {\em Schur double}.
The following theorem shows that under a natural assumption,    the subalgebra $D^A(n,d)\subseteq {}'{D^A(n,d)}$ is a {\em maximal} symmetric subalgebra:

\begin{MainTheorem}\label{TB}
Let $d\le n$ and $C$ be a subalgebra of\, ${}' {D^A(n,d)}$ such that 
$D^A(n,d)\subseteq C \subseteq {}'{D^A(n,d)}$. 
Suppose that for every prime $p$  
the $\F_p$-algebra 
$C\otimes_{\Z} \F_p$ is symmetric. Then $C=D^A(n,d)$. 
\end{MainTheorem}

Let $S^A(n,d):=\Inv^d M_n(A)$. If $A=\Z$, then $S^A(n,d)$ is just the (integral version of) the classical Schur algebra. The {\em generalized Schur algebras} $S^A(n,d)$ bear importance for the doubles, since, by Theorem~\ref{TA} and the easy observation that $T_{M_n(A)}\cong M_n(T_A)$, we can identify ${}'{D^A(n,d)}$ with $S^{T_A}(n,d)$ and $D^A(n,d)$ with an explicit subalgebra of $S^{T_A}(n,d)$. 

The superalgebras $S^A(n,d)$ can be studied using a generalized Schur-Weyl duality with the {\em super wreath product} $W^A_d:= A^{\otimes d}\rtimes \k\Si_d$. The superalgebra $M_n(A)$ can be identified with $\End_A(V)$, where $V:=A^{\oplus n}$. The following generalized version of Schur-Weyl duality is crucial for the proof of 
Turner's Conjecture, but is also of independent interest. 

\begin{MainTheorem}\label{TC}
The natural left $S^A(n,d)$-action and the natural right $W^A_d$-action on $V^{\otimes d}$ commute and yield an isomorphism $S^A(n,d)\cong \End_{W_d^A}(V^{\otimes d})$. 
\end{MainTheorem}

As a right $W_d^A$-supermodule, $V^{\otimes d}$ decomposes explicitly as a direct sum of certain {\em permutation supermodules} $M_\la^A$ where $\la$ runs over the set $\La(n,d)$ of all compositions of $d$ with $n$ parts. So Theorem~\ref{TC} realizes $S^A(n,d)$ as 
$$\End_{W_d^A}\bigg(\bigoplus_{\la\in\La(n,d)}M^A_\la\bigg).$$
For the purposes of Turner's Conjecture, it is important to `desuperize' this description of $S^A(n,d)$ in the case where $A$ is a certain {\em zigzag superalgebra} $\Zig$ depending on a  quiver $Q$. Let $|X|$ denote the algebra obtained from a superalgebra $X$ by forgetting the superstructure. 
We construct a (rather delicate) explicit isomorphism $\si$ from the ordinary wreath product $W_d^{|\Zig|}$ to $|W_d^\Zig|$. Twisting with this isomorphism makes the permutation module $M_\la^\Zig$ into an explicit {\em alternating sign permutation module} $M_\la^{|\Zig|}$ over $W_d^{|\Zig|}$. Then 
$$|S^\Zig(n,d)|\cong \End_{W_d^{|\Zig|}}\bigg(\bigoplus_{\la\in\La(n,d)}M^{|\Zig|}_\la\bigg).$$
Using Theorems~\ref{TA},\ref{TB},\ref{TC}, we obtain an explicit description of $D_Q(n,d)$ as a maximal symmetric subalgebra of the endomorphism algebra on the right hand side. This description is used in \cite{EK2} to identify $D_Q(n,d)$ with an algebra Morita equivalent to  (a $\Z$-form of) a RoCK block of a Hecke algebra or a more general cyclotomic KLR algebra, thus proving Turner's Conjecture.

Now we describe the contents of the paper in more detail. In Section~\ref{Sprel} we set up some basic combinatorial notation. In Section~\ref{SSuper} we discuss superspaces and superalgebras, especially symmetric and divided power superalgebras and various products and coproducts on them. In \S\ref{SSTEA} we consider trivial extension superalgebras. 
In Section~\ref{SD} we begin to study Turner doubles. The properties of invariant algebras $\Inv X$ are investigated in \S\ref{SSInv}. The definition of $D X$ is given in \S\ref{SSDoubles}, and its divided power version ${}'{DX}$ is studied in \S\ref{SSDivDoub}. 
For Theorem~\ref{TA} see Theorems~\ref{TFund} in \S\ref{SSGen} and \ref{Lgen} in \S\ref{SSGenSet}. 
We discuss gradings on doubles in~\S\ref{SSGr} and symmetricity 
of doubles in~\S\ref{SSDoubleSymm}.

Section~\ref{SXSW} is on generalized Schur-Weyl duality. 
In \S\ref{SSWreath} we discuss wreath product algebras  and permutation modules over them. In \S\ref{SSTensorSpace} we study the generalized tensor space, prove Theorem~\ref{TC} 
(see Lemma~\ref{LSchurIso}) and discuss connections with permutation modules over wreath product algebras. 
We consider idempotent truncations of generalized Schur algebras in \S\ref{SSIdempotentTruncation} and idempotent refinements of permutation modules in~\S\ref{SSColored}. Desuperization is discussed in \S\ref{SSDesup}. 

Section~\ref{SSymm} is on Schur doubles. 
In~\S\ref{SSGeneratingND} we identify $D^A(n,d)$ with the subalgebra of $S^{T_A} (n,d)$ generated by certain explicit elements. 
Theorem~\ref{TB} is proved in \S\ref{SSSymLat}, see Theorem~\ref{Tsym}. 
In \S\ref{SSBases} we discuss bases and product rules of Schur doubles and their divided power versions. Section~\ref{SQuiver} is on the important special case of the quiver Schur  (schiver) doubles. Quivers and zigzag algebras are considered in \S\ref{SSQ}. Finally, in~\S\ref{SSSchiver}, we discuss the degree zero component of a schiver double and results related to schiver generation and  desuperization, which will be needed in \cite{EK2}.

\section{Preliminaries}\label{Sprel}
Throughout the paper, $\k$ is an arbitrary commutative (unital) ring. In some constructions, involving divided powers, we will need to work over a more special ring $\O$, which is assumed to be a (commutative) integral domain with field of fractions $\K$ of {\em characteristic zero}. We assume that there is a fixed ring homomorphism $\O\to\k$, which allows us to extend scalars from $\O$ to $\k$, i.e.~to consider 
\[
V_{\k}:=V\otimes_\O\k
\] 
for any $\O$-module $V$. 
If $U$ and $V$ are $\k$-modules, we denote $U\otimes V: = U\otimes_\k V$.
Important examples of triples $(\K,\O,\k)$ are $(\Q,\Z,\F_p)$ and  $(\Q_p,\Z_p,\F_p)$.

\subsection{Weights and sequences}
\label{SSWeSe}
Let $n\in\Z_{>0}$ and $d\in \Z_{\ge 0}$. We denote by $\La(n)$ the set of compositions $\la=(\la_1,\dots,\la_n)$ with $n$ parts $\la_1,\dots,\la_n\in\Z_{\geq 0}$. We refer to the elements of $\La(n)$ as {\em weights}. 
For $\la=(\la_1,\dots,\la_n)\in\La(n)$, we denote $|\la|:=\la_1+\dots+\la_n$. We set 
$$\La(n,d):=\{\la\in\La(n)\mid |\la|=d\}.$$ 
More generally, if $S$ is a finite  set, we denote by $\La(S,d)$ the set of tuples $(\la_s)_{s\in S}$ of non-negative integers such that $\sum_{s\in S}\la_s=d$. 
For $S=[1,n]$, we identify $\La(S,d)$ with $\La(n,d)$.

For $1\leq m\leq n$, we have special weights 
$$\eps_m:=(0,\dots,0,1,0,\dots,0)\in\La(n,1),$$ 
with $1$ in the $m$th position, so that 
$$
\la=(\la_1,\dots,\la_n)=\la_1\eps_1+\dots+\la_n\eps_n.
$$

For $m,n\in\Z$, we consider the (possibly empty) {\em segments} 
\begin{align*}
[m,n]&:=\{r\in\Z\mid m\leq r\leq n\},  && (m,n]:=\{r\in\Z\mid m< r\leq n\},
\\
[m,n)&:=\{r\in\Z\mid m\leq r< n\}. &&
\end{align*}
The symmetric group $\Si_n$ acts naturally on the left on $[1,n]$.

Let $\Seq(n,d):=[1,n]^d$ be the set 
  of (ordered) $d$-tuples $\br=(r_1,\dots,r_d)$ where $r_1,\dots,r_d\in [1,n]$. The action of the symmetric group $\Si_d$ on $[1,d]$ yields the right action of $\Si_d$ on $\Seq(n,d)$ by place permutations: for $\br\in \Seq(n,d)$ and $g\in\Si_d$, we have $\br g=\bs$ where 
$
s_a=r_{ga}$ for all $a\in[1,d]$. 

For  $\la\in \La(n,d)$ we set 
 \begin{equation}\label{ELaSeq}
 {}^\la \Seq:=\{\br\in \Seq(n,d)\mid \eps_{r_1}+\dots+\eps_{r_d}=\la\}.
 \end{equation}
Then $\Seq(n,d)=\bigsqcup_{\la\in\La(n,d)}{}^\la \Seq$ is the decomposition of $\Seq(n,d)$ into  $\Si_d$-orbits. 

For $\la\in\La(n,d)$ we define 
  $$\br^\la:=(1,\dots,1,2,\dots,2,\dots, n,\dots,n)\in {}^\la\Seq,$$ where each $r\in[1,n]$ is repeated $\la_r$ times.

\subsection{Integer-valued matrices and sequences}
\label{SSPairsSeq}
Define $\Mat(n)$ to be the set of $n\times n$-matrices with non-negative integer coefficients. Let $E_{r,s}\in\Mat(n)$ denote the {\em matrix unit}\, with $1$ in the $(r,s)$th position. 
For $C=(c_{r,s})_{1\leq r,s\leq n}\in \Mat(n)$, we set $|C|:=\sum_{r,s=1}^n c_{r,s}$, and we define
$$\Mat(n,d):=\{C\in \Mat(n)\mid |C|=d\}.$$ 
Given $C,D\in\Mat(n)$,
define the integers
$$
C!=\prod_{r,s\in [1,n]}{c_{r,s}!},\qquad\quad {C\choose D}:=\prod_{r,s\in [1,n]}{c_{r,s}\choose d_{r,s}}.
$$
For any $C\in \Mat(n,d)$,
we further set
\begin{align*}
 \al(C)&:=\big(\textstyle\sum_s c_{1,s},\sum_s c_{2,s},\dots ,\sum_s c_{n,s}\big)\in\La(n,d),
\\
\be(C)&:=\big(\textstyle\sum_r c_{r,1},\sum_r c_{r,2},\dots ,\sum_r c_{r,n}\big)\in\La(n,d).
\end{align*}
Let $\la,\mu\in\La(n,d)$. Define 
\begin{align*}
{}_\mu\Mat(n,d)_\la&:=\{C\in\Mat(n,d)\mid \al(C)=\mu\ \text{and}\ \be(C)=\la\}.
\end{align*}
The subsets of $\Mat(n)$ and $\Mat(n,d)$ consisting of 
$\{0,1\}$-matrices are denoted by
\begin{align*}
\OO(n)&:=\{C\in \Mat(n)\mid c_{r,s}\in\{0, 1\}\ \text{for all $1\leq r,s\leq n$}\},
\\
 \OO(n,d)&:=\Mat(n,d)\cap\, \OO(n).
\end{align*}

In~\S\ref{SSBases},
we will use the following generalization. Let $\ttB=\ttB_\0\sqcup \ttB_\1$ be a set split as a disjoint union of two subsets $\ttB_\0$ and  $\ttB_\1$. Set   
\begin{equation}\label{EOBN}
\Mat^\ttB(n):=\{\bC=(C^\ttb)_{\ttb\in \ttB}\mid \text{$C^\ttb\in \Mat(n)$ for $\ttb\in \ttB_\0$, $\, C^\ttb\in \OO(n)$ for $\ttb\in \ttB_\1$}\}.
\end{equation}
Let $\bC=(C^\ttb)_{\ttb\in \ttB}\in \Mat^\ttB(n)$. For every $\ttb\in \ttB$, we write $C^\ttb=(c^\ttb_{r,s})_{1\leq r,s\leq n}$. Denote $|\bC|_\0:=\sum_{\ttb\in \ttB_\0}|C^\ttb|$, $|\bC|_\1:=\sum_{\ttb\in \ttB_\1}|C^\ttb|$,
\begin{align}
|\bC|&:=|\bC|_\0+|\bC|_\1=\sum_{\ttb\in \ttB}|C^\ttb|=\sum_{(r,s,\ttb)\in [1,n]^2\times \ttB}c^\ttb_{r,s},
\\
\Mat^\ttB(n,d)&:=\{\bC\in\Mat^\ttB(n)\mid |\bC|=d\}.
\label{EO^B(n,d)}
\end{align}
Let $\bC=(C^\ttb)_{\ttb\in\ttB}$ and $\bD=(D^\ttb)_{\ttb\in\ttB}\in \Mat^\ttB(n)$. We define $\bC+\bD$ by $(\bC+\bD)^\ttb=C^\ttb+D^\ttb$ for all $\ttb\in \ttB$. Note that $\bC+\bD$ may or may not be an element of $\Mat^\ttB(n)$. We set 
$$
\bC!:=\prod_{\ttb\in \ttB} C^\ttb!=\prod_{\ttb\in \ttB_\0} C^\ttb!,\quad 
{\bC\choose \bD}:=\prod_{\ttb\in \ttB}{C^\ttb\choose D^\ttb}.
$$
 
Define $\Seq^\ttB (n,d)^2$ to be the set of tuples 
\[
(\br, \ttbb, \bs) = ((r_1,\dots,r_d), (\ttb_1,\dots,\ttb_d), (s_1,\dots,s_d) ) \in \Seq(n,d) \times \ttB^d \times \Seq(n,d)
\]
such that for any distinct $k,l\in [1,d]$ with
$(r_k,\ttb_k,s_k)=(r_l,\ttb_l, s_l)$ 
we have $\ttb_k\in \ttB_\0$. 
The left action of $\Si_d$ on $[1,d]$ induces a right action on each component 
of the direct product
 $\Seq(n,d) \times \ttB^d \times \Seq(n,d)$ as in~\S\ref{SSWeSe}, 
 so we have a right action of $\Si_d$ on $\Seq^\ttB (n,d)^2$. There is a  bijection
 \[
 \Seq^\ttB(n,d)^2 /\Si_d \iso \Mat^\ttB (n,d), \; 
 (\br, \ttbb, \bs) \mapsto  \bM[\br,\ttbb,\bs]:=((c^{\ttb}_{r,s})_{r,s\in [1,n]})_{\ttb\in B}
 \]
 where 
\[ 
 c^{\ttb}_{r,s}= \sharp \{ k\in [1,d] \mid (r_k,\ttb_k, s_k) = (r,\ttb,s)\}.
 \]
We always identify $\Seq^\ttB(n,d)^2/\Si_d$ with $\Mat^\ttB(n,d)$ via this bijection. In particular, given $\bC\in \Mat^\ttB(n,d)$,
we write $(\br,\ttbb,\bs) \in \bC$ if $\bM[\br,\ttbb,\bs]=\bC$.

\subsection{Cosets}
\label{SSCosets}
Let $(S,<)$ be a totally ordered finite set.
Recall the notation $\La(S,d)$ from \S\ref{SSWeSe}.  
Let $\la=(\la_s)_{s\in S}\in\La(S,d)$. 
The corresponding {\em standard set partition} 
$\Omega^{\la}$ is the partition of $[1,d]$ into the segments 
$$
\Om^{\la}_s:=\big(\textstyle\sum_{t<s}\la_t, \sum_{t\leq s}\la_t\big]\qquad(s\in S).
$$
Note that the segment $\Om^{\la}_s$ has $\la_s$ elements. 
Write $S=\{s_1<\dots<s_n\}$. 
The {\em standard parabolic subgroup} 
\begin{equation}\label{EStandardParabolic}
\Si_{\la}\cong \Si_{\la_{s_1}}\times\dots\times\Si_{\la_{s_n}}\leq \Si_d
\end{equation}
preserves the set partition $\Om^{\la}$. 
If $\la\in \La(n,d)$, we define $\Om^\la$ and $\Si_\la$ via the usual total order on $[1,n]$.

Let $\la\in \La(S,d)$ and $\D^\la$ be the set of shortest coset representatives for $\Si_d/\Si_\la$, where the length $\ell(g)$ 
of an element 
$g\in \Si_{\la}$ is the smallest integer $\ell$ such that $g$ can be represented as a product of $\ell$ transpositions of the form $(r,r+1)$, $1\le r<d$.
For $\mu\in\La(S,d)$, we also have the set ${}^\mu\D$ of shortest coset representatives for $\Si_\mu\backslash\Si_d$ and the set ${}^\mu\D^\la$ of shortest double coset representatives for $\Si_\mu\backslash \Si_d/\Si_\la$. 
Note that we have a bijection
\begin{equation}\label{ESingleCosetSeq}
{}^\mu\D\to {}^\mu \Seq,\ \br^\mu \mapsto \br^\mu g 
\end{equation}
and a bijection ${}^\mu\D\to \D^\mu, g\mapsto g^{-1}$.

It  is well known and easy to see (cf.~e.g.~\cite[1.3.10]{JK}) that 
for every $C=(c_{r,s})\in {}_{\mu} \Mat(n,d)_{\la}$ there exists a unique element $g(C)\in {}^\mu \D^\la$ such that 
\[
|g(C) (\Om_s^\la)\cap \Om_r^\mu | = c_{r,s} 
\]
for all  $r,s\in [1,n]$.
Moreover:

\begin{Lemma} 
For any $\la,\mu\in\La(n,d)$, the map $C\mapsto g(C)$ defines a bijection
$$
\quad {}_\mu\Mat(n,d)_\la\iso {}^\mu\D^\la.
$$
\end{Lemma}

Given $C=(c_{r,s})\in{}_\mu\Mat(n,d)_\la$ and $1\leq s\leq n$, we have a composition
$$
\bc_{*,s}:=(c_{1,s},\dots,c_{n,s})\in\La(n,\la_s).
$$
Given elements $g_1\in\Si_{\la_1},\dots, g_n\in\Si_{\la_n}$, we consider $(g_1,\dots,g_n)\in\Si_{\la_1}\times\dots\times\Si_{\la_n}$ as an element of $\Si_d$ via the natural embedding of $\Si_{\la_1}\times\dots\times\Si_{\la_n}$ into $\Si_d$. 
Another easy and well-known result (see e.g.~\cite[Lemma 1.6]{DJ}) is:

\begin{Lemma} \label{LDJ} 
Let $\la,\mu\in\La(n,d)$. There is a bijection
\begin{align*}
\{(C,g_1,\dots,g_n)\mid C\in {}_\mu\Mat(n,d)_\la,\ g_s\in {}^{\bc_{*,s}}\D\ \text{for}\ s=1,\dots,n\} \iso {}^\mu \D
\end{align*}
defined by 
$(C,g_1,\dots,g_n)\mapsto g(C)(g_1,\dots,g_n).$
\end{Lemma}

\section{Superspaces and superalgebras}
\label{SSuper}
From now on, we write $\Z_2:=\Z/2\Z=\{\0,\1\}$.
Let $V=V_{\0}\oplus V_{\1}$ be a free $\k$-supermodule of finite rank. We refer to $V$ as a ($\k$-){\em superspace}. The $\k$-rank of $V$ is denoted by $\dim V$.
 For parities of elements, we write $\bar v=\0$ if $v\in V_\0$ and $\bar v=\1$ if $v\in V_\1$. 
Whenever $\bar v$ appears in a formula, this means that we assume that $v$ is a homogeneous element. 
If $V$ is an (associative unital) $\k$-superalgebra, we denote by $|V|$ the same algebra without the 
$\Z_2$-grading. 

By a {\em $\Z$-supergrading} on a superspace $V$ we mean a $\Z$-grading $V=\bigoplus_{m\in\Z} V^m$ such that $V^m=(V^m\cap V_\0)\oplus  (V^m\cap V_\1)$ for all $m\in\Z$.

\subsection{Dual superspaces and tensor products}
\label{SSSuper}

The dual $V^*:=\Hom_\k (V,\k)$ 
is a superspace in a natural way. 
We have the pairing $\langle \cdot, \cdot \rangle$ between $V$ and $V^*$:
$$
\lan  v,\be\ran=\lan \be, v\ran :=\be(v)\qquad(v\in V,\ \be\in V^*).
$$

Let $d\in \Z_{> 0}$, and $V_1,\dots,V_d$ be superspaces. The  tensor product $V_1\otimes \dots \otimes V_d$ is again a superspace in a natural way. 
We always identify $(V_1\otimes \dots \otimes V_d)^*$ with $V_1^*\otimes \dots\otimes V_d^*$ via
\begin{equation}\label{EDualTensor}
\langle \be_1\otimes\dots\otimes \be_d, v_1\otimes\dots\otimes v_d\rangle:=
(-1)^{[ \be_1,\dots, \be_d; v_1,\dots, v_d ]}
\langle\be_1,v_1\rangle \dots \langle \be_d,v_d\rangle,
\end{equation}
where $\be_a\in V_a^*, v_a\in V_a$ for $a=1,\dots,d$, and where 
\begin{equation}\label{ETPSign}
[ \be_1,\dots, \be_d; v_1,\dots, v_d ]:=\sum_{1\leq a<c\leq d} \bar \be_c \bar v_a
\end{equation}
is defined for (homogeneous) elements $\be_1,\dots, \be_d, v_1,\dots, v_d$ of arbitrary superspaces. Note that 
\begin{align*}
\langle \be_1\otimes\dots\otimes \be_d, v_1\otimes\dots\otimes v_d\rangle
&=\langle v_1\otimes\dots\otimes v_d, \be_1\otimes\dots\otimes \be_d\rangle
\\
&:=
(-1)^{[ v_1,\dots, v_d; \be_1,\dots, \be_d ]}
\langle v_1,\be_1\rangle \dots \langle v_d,\be_d\rangle,
\end{align*}
since $\langle v_a,\be_a\rangle=0$ unless $\bar v_a=\bar\be_a$ for any $1\leq a\leq d$. 

If $V_1,\dots,V_d$ are  $\k$-superalgebras,  
then $V_1\otimes\dots\otimes V_d$ is again a superalgebra with
$$
(v_1\otimes\dots\otimes v_d)(w_1\otimes\dots\otimes w_d)=
(-1)^{[ v_1,\dots, v_d; w_1,\dots, w_d ]}
v_1w_1\otimes\dots\otimes v_dw_d,
$$
for $v_a,w_a\in V_a$, $a=1,\dots,d$.

The symmetric group $\Si_d$ acts on the superspace $V^{\otimes d}$ on the right by (super) place permutations. More precisely, for $g\in\Si_d$ and $v_1,\dots, v_d\in V$, we define 
\begin{equation}\label{EGSign}
[ g;v_1,\dots,v_d]:=
\sum_{1\leq a<c\leq d,\ g^{-1}a>g^{-1}c}\bar v_a \bar v_c,
\end{equation}
and
\begin{equation}
\label{ESGAction}
(v_1\otimes\dots\otimes v_d)^g:=(-1)^{[ g;v_1,\dots,v_d]}v_{g1}\otimes\dots\otimes v_{gd}.
\end{equation}
If $V$ is a superalgebra, then $\Si_d$ acts on $V^{\otimes d}$ with algebra  automorphisms. 

\subsection{Symmetric and divided power superalgebras}
\label{SSSymDiv}
Recall that $\O$ is a domain of characteristic zero. 
Let $V=V_{\0}\oplus V_{\1}$ be an $\O$-superspace with bases $\ttB_\0=\{x_1,\dots,x_l\}$ of $V_{\0}$ and $\ttB_\1=\{x_{l+1},\dots,x_{l+m}\}$ of $V_{\1}$. Then $\ttB=\ttB_\0\sqcup \ttB_\1$ is a homogeneous basis of $V$. We identify $V_\k:=V\otimes_\O \k$ with the free $\k$-supermodule with basis $\ttB$, and we identify $V$ with the $\O$-subsupermodule $V\otimes 1\subseteq V_\K$. 

For every $d\in \Z_{\ge 0}$, consider the $\O$-superspace
\[
\Tens^d V:= V^{\otimes d}.
\]
Let
$$\Tens V:=\bigoplus_{d\in\Z_{\geq 0}}\Tens^d V$$ 
be the tensor superalgebra of $V$ and  $\Sym V=\bigoplus_{d\in\Z_{\geq 0}}\Sym^d V$ be the {\em symmetric superalgebra} on $V$. 
That is, $\Sym V$ is the quotient of $\Tens V$ by the ideal generated by all elements of the form 
$v\otimes u-(v\otimes u)^{(1,2)}$ for $u,v\in V$ and all elements of the form $v\otimes v$ for $v\in V_{\1}$. 
Moreover,  for every $d\in \Z_{\ge 0}$, the subsuperspace $\Sym^d V\le \Sym V$ is the intersection of $\Sym V$ with the subsuperspace $\Tens^d V$ of $\Tens V$.

We consider $\Sym V$ as an $\O$-form of $\Sym V_\K$. We will also need another $\O$-form. 
The {\em divided powers superalgebra}\, ${}'{\Sym V}=\bigoplus\, {}'{\Sym^d V}$ is the $\O$-subalgebra of $\Sym V_\K$ generated by the divided powers $v^{(m)}:=v^m/m!$ for all $v\in V_{\0}$ and $m\in\Z_{\geq 0}$ together with all $v\in V_{\1}$. We now define  
${}'{\Sym V_\k}:=({}'{\Sym V})\otimes_\O \k$ and write $v^{(m)}:=v^{(m)}\otimes 1\in {}'{\Sym V_\k}$. 

For every $d\in\Z_{\geq 0}$, we have the fixed points  
$
\Inv^d V:=\big(\Tens^d V)^{\Si_d}$ 
of the action (\ref{ESGAction}) and set $
\Inv V:=\bigoplus_{d\geq 0} \Inv^dV. 
$
It is a subalgebra of $\Tens V$ with respect to a new product, which we now define.

For $d,e\in\Z_{\geq 0}$, recall that ${}^{(d,e)}\D$ stands for the set of the shortest coset representatives for $(\Si_d\times\Si_e)\backslash \Si_{d+e}$.  We consider the linear map 
$$\Tens^d V\otimes  \Tens^e V\to \Tens^{d+e} V,\ t\otimes s\mapsto t*s,$$ 
defined by
\begin{equation}\label{EStar}
(x_1\otimes\dots\otimes x_d)*(y_1\otimes\dots\otimes y_e):=\sum_{g\in {}^{(d,e)}\D}(x_1\otimes\dots\otimes x_d\otimes y_1\otimes\dots\otimes y_e)^g
\end{equation}
for all $x_1,\dots,x_d,y_1,\dots,y_e\in V$. This new {\em $*$-product} (or {\em shuffle product}) on $\Tens V$ makes it an 
associative supercommutative superalgebra. Moreover, 
 $\Inv V$ is a subsuperalgebra of $\Tens V$ with respect to the $*$-product.

 Let $V = U\oplus W$ be a direct sum decomposition of $\O$-supermodules.
 For every $e\ge 0$, we identify $\Tens^e U$ and $\Tens^e W$ with subsupermodules of 
 $\Tens^e V$ in the obvious way. 
The following is easy to see: 

\begin{Lemma} \label{L120116} 
Let $d\in \Z_{\ge 0}$. 
For every $e\in [0,d]$, the $\O$-supermodule homomorphism 
\[
\Inv^e U \otimes \Inv^{d-e} W \to \Inv^d V, \; s\otimes t \mapsto s * t
\]
 is injective, and we have a direct sum decomposition of $\O$-superspaces:
\[
 \Inv^d V = \bigoplus_{e=0}^d\, (\Inv^e U) * (\Inv^{d-e} W).
\]
\end{Lemma}

To describe bases, set 
\begin{align*}
\Mat^\ttB:=\{(c_1,\dots,c_l,c_{l+1},\dots,c_{l+m})\mid c_1,\dots,c_l\in\Z_{\geq 0},\ c_{l+1},\dots,c_{l+m}\in\{0,1\}\}.
\end{align*}
For $\bc=(c_1,\dots,c_{l+m})$, define 
$
|\bc|:=c_1+\dots+c_{l+m}, 
$ 
and denote 
\begin{align*}
\Mat^\ttB_d:=\{\bc\in\Mat^\ttB \mid |\bc|=d\}.
\end{align*}
In terms of  (\ref{EOBN}), (\ref{EO^B(n,d)}), we have $\Mat^\ttB=\Mat^\ttB(1)$ and $\Mat^\ttB_d=\Mat^\ttB(1,d)$. 
Then 
\begin{equation}\label{EBasis1}
\{x_1^{c_1}\cdots x_{l+m}^{c_{l+m}}\mid (c_1,\dots,c_{l+m})\in\Mat^\ttB_d\}
\end{equation}
is a basis of $\Sym^d V$, 
\begin{equation}\label{EBasis2}
\{x_1^{(c_1)}\cdots x_{l+m}^{(c_{l+m})}\mid (c_1,\dots,c_{l+m})\in\Mat^\ttB_d\}
\end{equation}
is a basis of ${}'{\Sym^d V}$, and 
\begin{equation}\label{EBasis3}
\{x_1^{\otimes c_1} * \dots * x_{l+m}^{\otimes c_{l+m}}\mid(c_1,\dots,c_{l+m})\in\Mat^\ttB_d\}
\end{equation}
is a basis of $\Inv^dV$. 

Define
\[
\Star^d V :=\underbrace{V* \dots * V}_{d\ \text{times}}, \qquad  \Star V := \bigoplus_{d\ge 0} \Star^d V,
\]
so that $\Star V$ is an $\O$-subsupermodule of $\Inv V$.

\begin{Lemma} \label{LemmaTri} 
There is an isomorphism of algebras 
$\kappa \colon {}'{\Sym V}\iso \Inv V$ which maps 
$x_1^{(c_1)}\cdots x_{l_m}^{(c_{l+m})}$ to $ x_1^{\otimes c_1} * \dots * x_{l+m}^{\otimes c_{l+m}}$ for all $(c_1,\dots,c_{l+m})\in\Mat^\ttB_d$.
Moreover, $\kappa(\Sym(V))=\Star V$. 
\end{Lemma}

\begin{proof}
It follows easily from the definitions that 
there is a homomorphism of superalgebras $\Sym V\to \Inv V$ which is the identity on $V$.  Under this map, for any $(c_1,\dots,c_{l+m})\in\Mat^\ttB_d$, the basis element $x_1^{c_1}\dots x_{l_m}^{c_{l+m}}$ is sent to 
$$x_1^{* c_1} * \dots * x_{l+m}^{* c_{l+m}}=c_1!\dots c_{l+m}!\, x_1^{\otimes c_1} * \dots * x_{l+m}^{\otimes c_{l+m}}.$$ 
Extending scalars to $\K$ and restricting to ${}'{\Sym V}$, we obtain the desired isomorphism 
${}'{\Sym V}\iso \Inv V$. The final statement of the lemma is clear. 
\end{proof}

\subsection{Coproducts}\label{SCoprod}
We can also consider $
\Tens V
$ as a {\em supercoalgebra}, with the coproduct 
\begin{equation}\label{EBasicCoproduct}
\begin{split}
\De\colon 
\Tens^d V&\to 
\bigoplus_{e,f\geq 0,\ e+f=d}
\Tens^e V\otimes \Tens^f V,
\\ 
v_1\otimes\dots\otimes v_d &\mapsto \sum_{e,f\geq 0,\ e+f=d} (v_1\otimes\dots\otimes v_e)\otimes(v_{e+1}\otimes\dots\otimes v_d).
\end{split}
\end{equation}

For a supercoalgebra $(X,\Delta)$ and $x\in X$, we repeatedly use Sweedler's notation $$\De(x)=\sum x_{(1)}\otimes x_{(2)}
$$
where $x_{(1)}$ and $x_{(2)}$ are homogeneous whenever $x$ is. 

The following is a superalgebra version of the well-known fact (see e.g. \cite[Proposition 1.9]{Re}) that $\Tens V$ is a bialgebra with respect to $(*,\De)$:

\begin{Lemma}\label{LstDe} 
Let $s,t\in\Tens V$. Then
$$
\De(s*t)=\sum(-1)^{\bar s_{(2)}\bar t_{(1)}}(s_{(1)}*t_{(1)})\otimes (s_{(2)}*t_{(2)}). 
$$
\end{Lemma}
\begin{proof}
We may assume that $s=s_1\otimes \dots\otimes s_a$ and $t=t_1\otimes \dots\otimes t_b$ for some $s_1,\dots,s_a,t_1,\dots,t_b\in V$. 
Let $\pi^{e,f}$ be the projection from $\Tens V\otimes \Tens V$ onto the summand $\Tens^e V\otimes \Tens^f V$. Fix $e\in[0,a+b]$ and denote by $\sum_{(C,g_1,g_2)}$ the sum over all triples $(C,g_1,g_2)$ corresponding to taking $\la=(e,a+b-e)$ and $\mu=(a,b)$ in Lemma~\ref{LDJ}. Then using that lemma, we get 
\begin{align*}
\pi^{e,a+b-e}\De(s*t)=&\pi^{e,a+b-e}\De\sum_{h\in{}^{(a,b)}\D}(s\otimes t)^h
\\
=&\pi^{e,a+b-e}\De\sum_{(C,g_1,g_2)}(s\otimes t)^{g(C)(g_1,g_2)}
\\
=&\sum_{(C,g_1,g_2)}(-1)^m
(s_1\otimes\dots\otimes s_{c_{1,1}} \otimes t_1 \otimes \dots \otimes t_{c_{2,1}})^{g_1}
\\
&\otimes  (s_{c_{1,1}+1}\otimes\dots\otimes s_a\otimes t_{c_{2,1}+1}\otimes\dots\otimes t_b)^{g_2}
\\
=&\sum_{C}(-1)^m
\big((s_1\otimes\dots\otimes s_{c_{1,1}}) * (t_1 \otimes \dots \otimes t_{c_{2,1}})\big)
\\
&\otimes \big( (s_{c_{1,1}+1}\otimes\dots\otimes s_a) * (t_{c_{2,1}+1}\otimes\dots\otimes t_b)\big)
\\
=&\pi^{e,a+b-e}\sum(-1)^{\bar s_{(2)}\bar t_{(1)}}(s_{(1)}*t_{(1)})\otimes (s_{(2)}*t_{(2)}),
\end{align*}
where $m=(\bar t_1 + \dots + \bar t_{c_{2,1}})(\bar s_{c_{1,1}+1}+\dots+ \bar s_a)$. This completes the proof.
\end{proof}

Note that $\Inv V$ is a subsupercoalgebra of $\Tens V$.
The supercoalgebra $\Inv V$ is {\em supercocommutative}, i.e.\ if $\De(\xi)=\sum\xi_{(1)}\otimes\xi_{(2)}$ in Sweedler's notation for a (homogeneous) $\xi\in \Inv V$, then 
\begin{equation}\label{ESuperCoCo}
\De(\xi)=\sum(-1)^{\bar \xi_{(1)}\bar\xi_{(2)}}\xi_{(2)}\otimes\xi_{(1)}.
\end{equation}
Hence the (restricted) dual
$$
(\Inv V)^*:=\bigoplus_{d\geq 0} (\Inv^d V)^*
$$
has a superalgebra structure which is dual to the coalgebra structure on 
$\Inv V$. 
More precisely, the superbialgebra structure on $(\Inv V)^*$ is determined by the identity 
\[
\langle \xi \eta, x \rangle = \langle \xi \otimes \eta, \Delta(x) \rangle
\quad (\xi,\eta\in (\Inv V)^*, x\in \Inv X),
\]
where
as usual we identify $(\Inv V)^* \otimes (\Inv V)^*$ with $(\Inv V \otimes \Inv V)^*$ via~\eqref{EDualTensor}.
This makes $(\Inv V)^*$  a supercommutative superalgebra. Given $\xi_1,\dots,\xi_d\in V^*$, 
we have the functional $\xi_1\otimes\dots\otimes \xi_d\in (\Tens^d V)^*$. Extending by zero to the whole $\Tens V$ and restricting to $\Inv V$, we can interpret $\xi_1\otimes\dots\otimes \xi_d$ as an element of $(\Inv V)^*$. The following is now clear:

\begin{Lemma} \label{LemmaOdin}
The natural map $V^*\to (\Inv V)^*$ extends to the isomorphism of superalgebras $\Sym(V^*)\iso (\Inv V)^*$, which maps any product 
$\xi_1\cdots\xi_d\in \Sym^d(V^*)$ with $\xi_1,\dots,\xi_d\in V^*$ to the functional $\xi_1\otimes\dots\otimes \xi_d\in (\Inv V)^*$. 
\end{Lemma}

\begin{Corollary} \label{CPerfPair} 
Identifying the $\O$-submodule ${}'{\Sym(V^*)}\subseteq \Sym(V^*_\K)$ with an $\O$-submodule of $(\Inv V_\K)^*$ via Lemma~\ref{LemmaOdin}, we have 
$$
\Star V=\{x\in \Inv V_\K\mid \langle x, \xi\rangle\in\O\ \text{for all $\xi\in {}'{\Sym(V^*)}$}\}.
$$ 
\end{Corollary}
\begin{proof}
Recall the basis $\ttB=\{x_1,\dots,x_{l+m}\}$ of $V$, and let $\{\xi_1,\dots,\xi_{l+m}\}$ be the dual basis of $V^*$. By Lemma~\ref{LemmaTri}, 
$$
\{x^{*\bc}:=x_1^{* c_1}*\dots *x_{l+m}^{*c_{l+m}}\mid \bc=(c_1,\dots,c_{l+m})\in\Mat^\ttB\}
$$
is an $\O$-basis of $\Star V$. On the other hand,
$$
\{\xi^{(\bc)}:=\xi_1^{(c_1)}\cdots \xi_{l+m}^{(c_{l+m})}\mid \bc=(c_1,\dots,c_{l+m})\in\Mat^\ttB\}
$$
is an $\O$-basis of ${}'{\Sym(V^*)}$. It remains to note that
$\langle  x^{*\bc},\xi^{(\bd)}\rangle=\pm\de_{\bc,\bd}$. 
\end{proof}

\subsection{Trivial extension algebras}\label{SSTEA}
Let $A$ be a $\k$-superalgebra. We consider $A^*$ as an $A$-bimodule with respect to the {\em dual regular actions}  given by
\begin{equation}\label{EActions}
\langle \al\cdot a,b\rangle=\langle \al,a b\rangle,\ \langle b,a\cdot \al\rangle=\langle ba,\al\rangle\qquad(a,b\in A,\ \al\in A^*).
\end{equation}
We refer to this bimodule as the {\em dual regular superbimodule}. 

The {\em trivial extension superalgebra} $T_A$ of $A$ is 
$T_A=A\oplus A^*$ as a superspace, with multiplication 
\begin{equation}\label{ETrivInt}
(a,\al)(b,\be)=(ab,a\cdot \be+\al\cdot b)\qquad (a,b\in A,\ \al,\be\in A^*).
\end{equation}

Let $m\colon A\otimes A\to A$ be the multiplication map on $A$ and
$$
m^* \colon A^*\to A^*\otimes A^* 
$$ 
be the dual map. For $\al\in A^*$, we write $m^*(\al)=\sum \al_{(1)}\otimes \al_{(2)}$ using Sweedler's notation. Then
\begin{align*}
\langle bc,\al\rangle&=\langle b\otimes c,m^*(\al)\rangle
=\langle b\otimes c,\sum \al_{(1)}\otimes \al_{(2)}\rangle
=\sum (-1)^{\bar c\bar\al_{(1)}}\langle b,\al_{(1)}\rangle\langle c, \al_{(2)}\rangle,
\\
\langle \al,bc\rangle&=\langle m^*(\al),b\otimes c\rangle
=\langle \sum \al_{(1)}\otimes \al_{(2)},b\otimes c\rangle
=\sum (-1)^{\bar\al_{(2)}\bar b}\langle \al_{(1)},b\rangle\langle  \al_{(2)},c\rangle.
\end{align*}
Note that the right hand sides above are equal to each other since $\langle \al_{(1)},b\rangle=0$ unless $\bar\al_{(1)}=\bar b$ and $\langle  \al_{(2)},c\rangle=0$ unless $\bar \al_{(2)} = \bar c$.
The formulas imply that for any $a\in A$ and $\al\in A^*$, we have 
\begin{align}
a\cdot\al&=\sum(-1)^{\bar a\bar \al_{(1)}}\langle a,\al_{(2)}\rangle \al_{(1)},
\label{ELRegNew}
\\
\al\cdot a&=\sum(-1)^{\bar \al_{(2)}\bar a}\langle a,\al_{(1)}\rangle \al_{(2)}.
\label{ERRegNew}
\end{align}

Let $n\in\Z_{>0}$. The matrix algebra $M_n(A)$ is naturally a superalgebra. For $1\leq r,s\leq n$ and $a\in A$, 
the matrix $aE_{r,s}\in X$ with $a$ in the $(r,s)$th position and zeros  elsewhere will be denoted by
$\xi_{r,s}^a$. Then 
$
\xi_{r,s}^a\xi_{t,u}^b=\de_{s,t}\xi_{r,u}^{ab}.
$ 
We have $\overline {\xi_{r,s}^a}=\bar a$. For $\al\in A^*$ and $1\leq r,s\leq n$, we have the element $x_{r,s}^{\al}\in M_n(A)^*$ defined from
\begin{equation}\label{EXRSAl}
\langle x_{r,s}^{\al},\xi_{t,u}^a\rangle=\de_{r,t}\de_{s,u}\langle \al,a\rangle\qquad (1\leq t,u\leq n,\ a\in A).
\end{equation}

\begin{Lemma} \label{LTMx} 
There is an isomorphism of superalgebras
\begin{align*}
M_n(T_A)\iso T_{M_n(A)},\ \xi_{r,s}^{(a,\al)}\mapsto (\xi_{r,s}^a, x_{s,r}^\al) 
\end{align*}
for all $1\leq r,s\leq n$, $a\in A$ and $\al\in A^*$. 
\end{Lemma}
\begin{proof}
Let $1\leq r,s,t,u,v,w\leq n$ and $a,b,c\in A$. 
On the one hand, we have 
\begin{align*}
\xi_{r,s}^{(a,\al)}\xi_{t,u}^{(b,\be)}=\de_{s,t}\xi_{r,u}^{(a,\al)(b,\be)}
=\de_{s,t}\xi_{r,u}^{(ab,a\cdot\be+\al\cdot b)}
\mapsto \de_{s,t}(\xi_{r,u}^{ab},x_{u,r}^{a\cdot\be+\al\cdot b}).
\end{align*}
On the other hand, 
\begin{align*}
(\xi_{r,s}^a, x_{s,r}^\al)(\xi_{t,u}^b, x_{u,t}^\be)=
(\xi_{r,s}^a\xi_{t,u}^b,  \xi_{r,s}^a\cdot x_{u,t}^\be + x_{s,r}^\al\cdot \xi_{t,u}^b). 
\end{align*}
Since $\xi_{r,s}^a\xi_{t,u}^b=\de_{s,t}\xi_{r,u}^{ab}$, we just need to prove that 
\begin{equation}\label{E080116}
\xi_{r,s}^a\cdot x_{u,t}^\be + x_{s,r}^\al\cdot \xi_{t,u}^b=\de_{s,t}x_{u,r}^{a\cdot\be+\al\cdot b}.
\end{equation} 
But  
\begin{align*}
(\xi_{r,s}^a\cdot x_{u,t}^\be + x_{s,r}^\al\cdot \xi_{t,u}^b)(\xi_{v,w}^c)&=
x_{u,t}^\be(\xi_{v,w}^c\xi_{r,s}^a)+x_{s,r}^\al( \xi_{t,u}^b\xi_{v,w}^c)
\\
&=\de_{w,r}x_{u,t}^\be(\xi_{v,s}^{ca})+\de_{u,v}x_{s,r}^\al( \xi_{t,w}^{bc})
\\
&=\de_{w,r}\de_{u,v}\de_{t,s}\langle\be,ca\rangle+\de_{u,v}\de_{s,t}\de_{r,w}\langle\al,bc\rangle
\\
&=\de_{s,t}\de_{u,v}\de_{r,w}(a\cdot\be+\al\cdot b)(c)
\\
&=\de_{s,t}x_{u,r}^{a\cdot\be+\al\cdot b}(\xi_{v,w}^c),
\end{align*}
proving (\ref{E080116}).
\end{proof}

\section{Turner doubles}\label{SD}

In this section, we review and develop Turner's theory of doubles  \cite{Turner,TurnerT,TurnerCat}. We will freely use the notation and conventions of Section~\ref{SSuper}.  
Let $X$ be an  $\O$-superalgebra, free of  finite rank as an  $\O$-supermodule. We consider $X_\k=X\otimes_\O\k$ as a $\k$-superalgebra. 

\subsection{Invariants}\label{SSInv}

For $d\in\Z_{\geq 0}$ we have a superalgebra structure on $\Tens^d X:=X^{\otimes d}$ induced by that on $X$. 
So we have a (locally-unital) superalgebra structure on 
$\Tens X:=\bigoplus_{d\geq 0} \Tens^d X$, 
with the product on each summand $\Tens^d X$ being as above, and $xy=0$ for $x\in \Tens^d X$ and $y\in \Tens^e X$ with $d\neq e$. Note that this algebra structure is different from the two  algebra structures on $\Tens X$ considered in \S\ref{SSSymDiv}, namely the product $\otimes$ and the product $*$.

In  fact, $
\Tens X
$ is now even a {\em superbialgebra}\, with the coproduct (\ref{EBasicCoproduct}). Since $\Si_d$ acts on $\Tens^d X$ with superalgebra automorphisms, the fixed points 
$
\Inv^d X=(\Tens^d X)^{\Si_d}
$
 is a subsuperalgebra of $\Tens^d X$. By observations made in \S\ref{SCoprod},  
$
\Inv X=\bigoplus_{d\geq 0} \Inv^dX
$
is a supercocommutative subsuperbialgebra of $\Tens X$.

\begin{Lemma}\label{Lstfoprem}
Let $x,y\in \Tens X$ and $z\in \Inv X$.  Then  
\begin{align*}
 (x*y) z &= \sum (-1)^{\bar y \bar z_{(1)}} (xz_{(1)}) * (y z_{(2)}), \\
 z (x*y) &= \sum (-1)^{\bar z_{(2)} \bar x} (z_{(1)} x) * (z_{(2)} y). 
\end{align*}
\end{Lemma}

\begin{proof}
We may assume that $z\in \Inv^d X$, $x\in \Tens^e X$ and $y\in \Tens^{d-e} X$ for some non-negative integers $d\ge e$. 
Write $\sum' z_{(1)} \otimes z_{(2)}$ for the $\Inv^e X \otimes \Inv^{d-e} X$-component of 
$\De(z)$. 
Then, since $z$ is $\Si_d$-invariant, we have
\begin{align*}
(x*y) z & = \sum_{g\in\, {}^{(e,d-e)}\D} (x\otimes y)^g z \\
&= \sum_{g\in\, {}^{(e,d-e)}\D} ((x\otimes y) z)^g \\
&=  \sum_{g\in\, {}^{(e,d-e)}\D} 
\Big( \textstyle\sum' \big( (x\otimes y) (z_{(1)} \otimes z_{(2)}) \big) \Big)^g \\
&= \sum_{g\in\, {}^{(e,d-e)}\D} \Big( \textstyle\sum' (-1)^{\bar y \bar z_{(1)}} xz_{(1)}\otimes y z_{(2)} \Big)^g \\
&=
\sum\nolimits' (-1)^{\bar y \bar z_{(1)}} (xz_{(1)}) * (yz_{(2)})\\
&= \sum\nolimits (-1)^{\bar y \bar z_{(1)}} (xz_{(1)}) * (yz_{(2)}),
\end{align*}
where the last equality holds because a summand on the right hand side is zero unless 
$z_{(1)} \in \Inv^e X$ and $z_{(2)} \in \Inv^{d-e} X$. 
The second equality in the lemma is proved similarly. 
\end{proof}

\begin{Lemma}\label{Lstfo}
Let $x,y,z,u\in \Inv X.$ Then 
\[
(x*y) (z*u) = \sum (-1)^{s}
  (x_{(1)} z_{(1)}) *  (y_{(1)}z_{(2)}) * (x_{(2)} u_{(1)}) * (y_{(2)} u_{(2)}),
\]
where $s=(\bar x_{(2)}+\bar y_{(2)}) \bar z + \bar y_{(1)} (\bar x_{(2)} +  \bar z_{(1)}) 
+ \bar y_{(2)} \bar u_{(1)}$. 
\end{Lemma}

\begin{proof}
 Writing $\De(x*y)=\sum (x*y)_{(1)} \otimes (x*y)_{(2)}$,
we have
 \begin{align*}
 (x*y) (z*u) &= \sum (-1)^{\overline{(x*y)_{(2)}}\bar z} ((x*y)_{(1)} z) * ((x*y)_{(2)} u) \\
 &=\sum (-1)^{(\bar x_{(2)}+\bar y_{(2)}) \bar z + \bar y_{(1)} \bar x_{(2)}}
 ( (x_{(1)}*y_{(1)}) z) * ((x_{(2)}*y_{(2)})u) \\
 &= 
 \sum (-1)^s
 (x_{(1)} z_{(1)}) * (y_{(1)} z_{(2)}) * (x_{(2)} u_{(1)}) * (y_{(2)} u_{(2)} ),
 \end{align*}
 where $s = (\bar x_{(2)}+\bar y_{(2)}) \bar z + \bar y_{(1)} \bar x_{(2)} + \bar y_{(1)} \bar z_{(1)} + \bar y_{(2)} \bar u_{(1)}$ is as in the statement of the lemma, the first and third equalities hold by Lemma~\ref{Lstfoprem}, and the second one is due to Lemma~\ref{LstDe}.
\end{proof}

\begin{Lemma} \label{LExercise} 
Let $l\in\Z_{>0}$, $d_1,\dots,d_l,f_1,\dots,f_l\in\Z_{\geq 0}$, and $1_X=e_1+\dots+e_l$ with $e_ie_j=\de_{i,j}e_i$ for all $i,j$. If $x_i\in(\Tens^{d_i} X) e_i^{\otimes d_i}$ and $y_i\in e_i^{\otimes f_i}(\Tens^{d_i}  X) $ for $i=1,\dots,l$, then 
$$
(x_1*\dots*x_l)(y_1*\dots*y_l)=(-1)^{[x_1,\dots,x_l;y_1,\dots,y_l]}\de_{d_1,f_1}\dots\de_{d_l,f_l}(x_1y_1)*\dots*(x_ly_l).
$$
\end{Lemma}
\begin{proof}
Let $\la=(d_1,\dots,d_l)$ and $\mu=(f_1,\dots,f_l)$. Note that 
$(x_1\otimes\dots\otimes x_l)^g(y_1\otimes\dots\otimes y_l)^h=0
$
if $g\in{}^\la\D$, $h\in{}^\mu\D$ and either $\la\neq \mu$ or $g\neq h$. Since $\Si_d$ acts on $\Tens^d X$ with superalgebra automorphisms for every $d$, the result follows.  
\end{proof}

\begin{Corollary} \label{CExercise} 
If $X= X_1\oplus\dots\oplus X_l$ is a direct sum of superalgebras, then there is an isomorphism of superalgebras
$$
\bigoplus_{(d_1,\dots,d_l)\in \La(l,d)}\Inv^{d_1} X_1\otimes \dots\otimes  \Inv^{d_l} X_l\iso \Inv^d X,\ x_1\otimes\dots\otimes x_l\mapsto x_1*\dots*x_l.
$$
\end{Corollary}
\begin{proof}
This follows from Lemmas~\ref{L120116} and \ref{LExercise}.
\end{proof}

Recall from from \S\ref{SSTEA} that we consider $X^*$ as a bimodule over $X$. Note for $d\in\Z_{\geq 0}$ that $\Tens^d (X^*)$ is naturally a bimodule over $\Tens^d X$ with respect to
\begin{equation}\label{EdAction}
(x_1\otimes \dots\otimes x_d)\cdot (\xi_1\otimes\dots\otimes \xi_d)=(-1)^{[x_1,\dots,x_d;\xi_1,\dots,\xi_d]}
(x_1\cdot\xi_1)\otimes \dots\otimes (x_d\cdot \xi_d),
\end{equation}
where $x_1,\dots,x_d\in X$ and $\xi_1,\dots,\xi_d\in X^*$, or $\xi_1,\dots,\xi_d\in X$ and $x_1,\dots,x_d\in X^*$. 
As usual, if $d\neq e$ we define the action trivially: $\Tens^d X\cdot \Tens^e(X^*)=\Tens^e(X^*)\cdot \Tens^d X=0$. This yields a $\Tens X$-bimodule structure on $\Tens (X^*)$. 
Upon restriction, we now get an $\Inv X$-superbimodule structure on $\Inv (X^*)$. We refer to this superbimodule structure as the 
{\em  standard superbimodule} structure. On the other hand, we have the dual regular $\Inv X$-superbimodule structure on $(\Inv X)^*$, see 
 (\ref{EActions}). 
By Lemmas~\ref{LemmaOdin} and \ref{LemmaTri}, we have an embedding
\begin{equation}\label{EIota}
\iota\colon (\Inv X)^*\iso \Sym(X^*)\,\, \into\,\, {}'{\Sym(X^*)}\iso \Inv(X^*).
\end{equation}

\begin{Lemma}\label{Liotahom} 
The embedding $\iota$ is a homomorphism of $\Inv X$-bimodules. 
\end{Lemma}
\begin{proof}
Every element of $\Inv(X^*)$ is by definition a linear combination of functions of the form $\xi_1\otimes\dots\otimes \xi_d$ with $\xi_1,\dots, \xi_d\in X^*$. 
On the other hand, by Lemma~\ref{LemmaOdin}, $(\Inv X)^*$ is spanned by the functions of the form $(\xi_1\otimes\dots\otimes \xi_d)|_{\Inv X}$ with $\xi_1,\dots, \xi_d\in X^*$, and 
$$
\iota((\xi_1\otimes\dots\otimes \xi_d)|_{\Inv X})=\xi_1*\dots*\xi_d.
$$ 
Note that 
$$
(\xi_1*\dots*\xi_d)|_{\Inv X}=d!(\xi_1\otimes\dots\otimes \xi_d)|_{\Inv X}.
$$
We have proved for any $\xi\in \Tens^d(X^*)$ that 
\begin{equation}\label{E130116}
\iota(\xi|_{\Inv X})|_{\Inv X}=d!\xi|_{\Inv X}.
\end{equation}

Let $x\in\Inv X$. We now prove that $\iota(x\cdot (\xi|_{\Inv X}))=x\cdot\iota( \xi|_{\Inv X})$, the proof that $\iota((\xi|_{\Inv X})\cdot x)=\iota( \xi|_{\Inv X})\cdot x$ being similar. Using (\ref{E130116}), we get
\begin{align*}
\iota(x\cdot (\xi|_{\Inv X}))|_{\Inv X}
&=\iota((x\cdot \xi)|_{\Inv X})|_{\Inv X}
=d!(x\cdot \xi)|_{\Inv X}
=d!\,x\cdot (\xi|_{\Inv X})
\\
&=x\cdot\big((\iota( \xi|_{\Inv X}))|_{\Inv X}\big)
=(x\cdot\iota( \xi|_{\Inv X}))|_{\Inv X}.
\end{align*}

To prove that $\iota(x\cdot (\xi|_{\Inv X}))=x\cdot\iota( \xi|_{\Inv X})$ it now suffices to show that the map $\Inv(X^*)\to (\Inv X)^*$ given by $\eta\mapsto \eta |_{\Inv X}$ is injective. Let $\eta\in \Inv^d(X^*)$ satisfy $\eta |_{\Inv X}=0$. Since $d!({}'{\Sym(X^*)})\subseteq \Sym(X^*)$, we can write $d!\eta=\iota(\xi|_{\Inv X})$ for some $\xi\in \Tens^d(X^*)$. Then, using (\ref{E130116}), 
$$
0=d!\eta|_{\Inv X}=\iota(\xi|_{\Inv X})|_{\Inv X}
=d!\xi|_{\Inv X}
$$
Hence $\xi|_{\Inv X}=0$. 
But $\iota(\xi|_{\Inv X})=d!\eta$, 
whence $\eta=0$, as desired.
\end{proof}

Recall  
the trivial extension algebra $T_X=X\oplus X^*$ from \S\ref{SSTEA}.
For $d,e\in\Z_{\geq 0}$, we define $\Tens^{d,e} T_X$ to be the span in $\Tens^{d+e} T_X$ of pure tensors $y_1\otimes \dots\otimes y_{d+e}$ such that $d$ of the $y$'s are in $X$ and $e$ of the $y$'s are in $X^*$. 
We identify $\Tens^d X$ with $\Tens^{d,0} T_X$ and $\Tens^d (X^*)$ with $\Tens^{0,d} T_X$ in the obvious way. 
Then for $\xi\in \Tens^d X$ and $x\in \Tens^e (X^*)$, we have
$$
\xi x=\xi\cdot x\quad \text{and}\quad x\xi= x\cdot \xi,
$$
where the left hand sides are products in the algebra $\Tens^d T_X$ and the right hand sides are the {\em standard} actions in the sense of (\ref{EdAction}). (Note the change of our notational `paradigm': from now on we use Greek letters to denote elements of $\Tens X$ and Roman letters for elements of
$\Tens (X^*)$.)

\begin{Lemma} \label{L140116_1} 
Let $a,b,d\in\Z_{\geq 0}$ with $a,b\leq d$. Suppose that
$x\in\Inv^a (X^*)$, $y\in\Inv^b (X^*)$, 
$\xi\in\Inv^{d-a}X$, and $\eta\in\Inv^{d-b}X$. Then in $\Inv^d T_X$ we have
$$
(\xi*x)(\eta*y)=\sum (-1)^{\bar\xi_{(1)}(\bar\xi_{(2)}+\bar\eta+\bar x)+\bar \eta_{(1)}\bar x} 
 \xi_{(2)}\eta_{(1)} * (x\cdot\eta_{(2)})*(\xi_{(1)}\cdot y).
$$
\end{Lemma}

\begin{proof}
Since $X^* X^*=0$ in $T_X$, the result follows from ~\eqref{ESuperCoCo} and Lemma~\ref{Lstfo}.
\end{proof}

\subsection{Doubles}\label{SSDoubles}
We have a natural pairing $\langle \cdot, \cdot \rangle$ between $\Inv X$ and $(\Inv X)^*$, with $\langle x, \xi \rangle=\langle \xi, x \rangle=0$ for $\xi\in\Inv^d X$ and $x\in (\Inv^e X)^*$ with $d\neq e$. Also, for every $d\in\Z_{\geq 0}$ we have the {\em dual regular} actions (\ref{EActions}) of $\Inv^d X$ on $(\Inv^d X)^*$. Again, we declare that $\xi\cdot x=x\cdot \xi=0$ if $\xi\in\Inv^d X$ and $x\in (\Inv^e X)^*$ with $d\neq e$. 
There is a superbialgebra structure on $(\Inv X)^*$ which is dual to that on $\Inv X$. We write 
\begin{equation}\label{ENabla}
\nabla\colon (\Inv X)^*\to (\Inv X)^*\otimes (\Inv X)^*
\end{equation}
for the corresponding coproduct. 
Note that $\nabla((\Inv^d X)^*)\subseteq (\Inv^d X)^* \otimes (\Inv^d X)^*$ for all $d\in \Z_{\ge 0}$. 

We now recall Turner's construction \cite{TurnerCat} of a {\em double superalgebra} $DX$. As an $\O$-superspace, 
$$
DX:=\Inv X\otimes (\Inv X)^*.
$$
The product is defined, using Sweedler's notation for $\Delta$, as follows:
\begin{equation}\label{EProductDoubleBasic}
(\xi\otimes x)(\eta\otimes y)=\sum (-1)^{\bar\xi_{(1)}(\bar\xi_{(2)}+\bar\eta+\bar x)+\bar \eta_{(1)}\bar x} 
 \xi_{(2)}\eta_{(1)}\otimes (x\cdot\eta_{(2)})(\xi_{(1)}\cdot y)
\end{equation}
for $\xi,\eta\in \Inv X$ and $x,y\in  (\Inv X)^*$. 
The associativity of the product can be checked by a straightforward 
computation, cf.~\cite[Theorem 1.1]{TurnerCat}.
In view of (\ref{ELRegNew}) and (\ref{ERRegNew}), this product formula can be rewritten, using Sweedler's notation for $\Delta$ and $\nabla$, to match \cite[Remark 1.3]{TurnerCat}:
\begin{equation}\label{EProductDoubleGen}
(\xi\otimes x)(\eta\otimes y)=\sum (-1)^s 
\langle \xi_{(1)},y_{(2)}\rangle
\langle x_{(1)},\eta_{(2)}\rangle
 \xi_{(2)}\eta_{(1)}\otimes x_{(2)}y_{(1)},
\end{equation}
where 
$$
s=\bar\xi_{(1)}\bar\xi_{(2)}+\bar\xi_{(1)}\bar\eta_{(1)}+\bar x_{(2)}\bar y_{(2)}+\bar y_{(1)}\bar y_{(2)} + \bar x_{(1)}\bar \eta_{(1)}+\bar x_{(2)}\bar \eta_{(2)}+\bar x_{(2)}\bar \eta_{(1)}.
$$

It is easy to see that we can write the superalgebra $DX$ as a direct sum of subsuperalgebras
$$
DX=\bigoplus_{d\geq 0} D^d X,
$$
where
\begin{equation}\label{EDirectSum}
D^dX:=\bigoplus_{e,f\geq 0,\ e+f=d}\Inv^e X\otimes(\Inv^f X)^*.
\end{equation}
We use the following notation for the summands on the right hand side above:
\begin{equation}\label{EDEF}
D^{e,f}X:=\Inv^e X\otimes(\Inv^f X)^*.
\end{equation}

\begin{Remark} 
{\rm 
The definition of the double $D^d X$ makes sense for any $\k$-algebra $X$, without any assumption on the ring $\k$. We also note that Lemmas~\ref{Lstfoprem}, \ref{Lstfo}, and \ref{L140116_1} do not need the assumption that $\k=\O$. However, it is crucial to work over $\O$ when we deal with the divided power version ${}'{D^d X}$ below. 
}
\end{Remark}

\begin{Remark} \label{R123}
{\rm The direct sum decomposition in (\ref{EDirectSum}) is a priori only a  decomposition of $\O$-modules. But one can say a little more. 

(i) $D^{d,0} X$ is a subalgebra of $D^dX$ naturally isomorphic to the algebra $\Inv^dX$. 

(ii) $D^{0,d} X$ is an ideal  in $D^dX$. 
Moreover, $$(D^{e,f}X )( D^{0,d} X)   = (D^{0,d} X)( D^{e,f}X)=0$$ unless $e=d$, in which case for $\xi\in \Inv^dX$ and $x\in (\Inv^d X)^*$, we have
\begin{align}
(\xi\otimes 1)(1\otimes x)&=1\otimes (\xi\cdot x),
\label{ELReg}
\\
(1\otimes x)(\xi\otimes 1)&=1\otimes (x\cdot \xi).
\label{ERReg}
\end{align}
In particular, $D^{d,0} X\oplus D^{0,d} X$ is a subalgebra of $D^dX$, isomorphic to $T_{\Inv^d X}$. 
As a still more special case, we get $D^1X\cong T_X$.
}
\end{Remark}

\subsection{Divided power doubles}\label{SSDivDoub}
In view of Lemma~\ref{LemmaOdin}, we identify the superalgebras 
\begin{equation}\label{EId}
(\Inv X)^*=\Sym(X^*).
\end{equation} 
Then 
$$
\Sym(X^*)\subseteq {}'{\Sym(X^*)}\subseteq \Sym(X^*)\otimes_\O\K\cong \Sym(X^*_\K)=(\Inv X_\K)^*,
$$
where we have used the identification (\ref{EId}) over $\K$ for the last equality. We have the left and right dual regular actions of $\Inv X_{\K}$ on $(\Inv X_{\K})^*$. Since $\Inv X\subseteq \Inv X_{\K}$ in a natural way, we can also speak of the dual regular actions of $\Inv X$ on $(\Inv X_{\K})^*$. 

\begin{Lemma} \label{L140116_2} 
The $\O$-submodule ${}'{\Sym(X^*)}\subset (\Inv X_{\K})^*$ 
is invariant with respect to the dual regular actions of $\Inv X$ on $(\Inv X_{\K})^*$. Thus, ${}'{\Sym(X^*)}$ is an $\Inv X$-superbimodule. With respect to this $\Inv X$-superbimodule structure on ${}'{\Sym(X^*)}$ and the standard $\Inv X$-superbimodule structure on $\Inv(X^*)$, the map 
$\kappa\colon {}'{\Sym (X^*)} \iso \Inv(X^*)$ of Lemma~\ref{LemmaTri}
 is an isomorphism of $\Inv X$-superbimodules. 
\end{Lemma}
\begin{proof}
By (\ref{EIota}) and Lemma~\ref{Liotahom}, we have an $\Inv X$-bimodule homomorphism 
\begin{equation}\label{E190216}
\iota\colon (\Inv X)^*=\Sym(X^*)\,\into\, {}'{\Sym(X^*)}\stackrel{\kappa}{\longrightarrow}\Inv(X^*).
\end{equation}
As 
$$\Sym(X_\K^*)\cong\Sym(X^*)\otimes_\O\K\cong {}'{\Sym(X^*)}\otimes_\O\K\cong {}'{\Sym(X_\K^*)},$$ 
extending scalars in (\ref{E190216}), we get an $\Inv X_{\K}$-superbimodule isomorphism 
$$\iota_{\K} \colon (\Inv X_\K)^*=\Sym(X_\K^*)\,=\, {}'{\Sym(X_\K^*)}\iso\Inv(X_\K^*).$$ 
Considering ${}'{\Sym(X^*)}$ as the sublattice in $\Sym(X_\K^*)$, the restriction 
$\iota_\K|_{{}'{\Sym(X^*)}}$ is the isomorphism $\kappa\colon {}'{\Sym (X^*)} \iso \Inv(X^*)$.  
Now the standard left and right actions of $\Inv X\subseteq \Inv X_\K$ on $\Inv (X^*_{\K})$ leave 
 $\Inv (X^*)=\Inv (X^*_{\K})\cap \Tens (X^*)$  invariant, and we have 
$\iota_{\K}^{-1} (\Inv (X^*)) = {}'{\Sym(X^*)}$. 
This implies the lemma.
\end{proof}

The identification $\Sym(X^*_{\K})=(\Inv X_\K)^*$ from  (\ref{EId}) together with the coproduct~\eqref{ENabla}
yield a coproduct 
$$
\nabla_{\K} \colon \Sym(X^*_{\K}) \to \Sym(X^*_{\K}) \otimes \Sym(X^*_{\K}).
$$ 

\begin{Lemma}\label{LDeinv}
We have 
\[
 \nabla_{\K} ({}'{\Sym (X^*)}) \subseteq \big( \Sym(X^*) \otimes {}'{\Sym (X^*)} \big) \cap
 \big( {}'{\Sym(X^*)} \otimes \Sym (X^*) \big).
\]
\end{Lemma}

\begin{proof}
Let $x\in {}'{\Sym^d (X^*)}$ for some $d\in \Z_{\ge 0}$. 
Let $\{ \xi_1,\ldots,\xi_m\}$ be a homogeneous basis of $\Inv^d X$ and 
$\{ x_1,\ldots,x_m\}$ be the dual basis of $(\Inv^d X)^*=\Sym^d(X^*)$, cf. (\ref{EId}). We can write $\nabla_{\K} (x) = \sum_{j=1}^m y_j \otimes x_j$, where $y_j\in \Sym^d (X^*_{\K})$ for $j=1,\dots,m$. 
By Lemma~\ref{L140116_2}, 
${}'{\Sym(X^*)}$ is invariant under the left dual regular action of $\Inv X$, so
$\xi_i \cdot x \in {}'{\Sym^d (X^*)}$ for any 
$i\in \{1,\dots,m\}$. On the other hand, by~\eqref{ELRegNew}, 
\[
\xi_i \cdot x = \sum_{j=1}^m (-1)^{\bar \xi_i \bar y_j} \langle \xi_i,x_j\rangle y_j
= (-1)^{\bar \xi_i \bar y_i} y_i,
\]
whence $y_i\in {}'{\Sym^d (X^*)}$. We have proved that 
$\nabla_{\K} ({}'{\Sym (X)}) \subseteq \big( {}'{\Sym(X^*)} \otimes \Sym (X^*) \big)$.
The other inclusion is proved similarly. 
\end{proof}


By Lemma~\ref{LDeinv}, we have a coproduct 
\begin{equation}\label{ENewDe}
\nabla\colon {}'{\Sym(X^*)} \to {}'{\Sym(X^*)} \otimes {}'{\Sym (X^*)}
\end{equation} 
obtained by restricting $\nabla_{\K}$. 
Recalling (\ref{EId}), note that 
\begin{equation}\label{EDXZ}
DX=\Inv X\otimes (\Inv X)^*=\Inv X\otimes \Sym (X^*)
\end{equation}
is an $\O$-form of $DX_\K$. 
We define a larger $\O$-form
$$
{}'{DX}:=\Inv X\otimes {}'{\Sym(X^*)},
$$
which is closed under the multiplication (\ref{EProductDoubleBasic}) because 
$'{\Sym (X^*)}$ is invariant under the left and right dual regular actions of $\Inv X$ by Lemma~\ref{L140116_2}. The product in ${}'{DX}$ is also given by the 
formula~\eqref{EProductDoubleGen} where we use the coproduct 
(\ref{ENewDe}) on $x$ and $y$.   
We have ${}'{D X} = \bigoplus_{d\ge 0} {}'{D^d X}$, where 
\[
{}'{D^d X} := \sum_{e,f\ge 0, \; e+f=d} \Inv^e X \otimes {}'{\Sym^f (X^*)}. 
\]
We use the following notation for the summands on the right hand side above:
\begin{equation}\label{EDEFPrime}
{}'{D^{e,f}X}:=\Inv^e X\otimes {}'{\Sym^f (X^*)}.
\end{equation}

The following result often allows one to reduce the study of $D X$ to that of $\Inv T_X$. Recall the isomorphism $\kappa$ from Lemma~\ref{LemmaTri}. 

\begin{Theorem} \label{TFund} 
There is an isomorphism of $\O$-superalgebras 
\[
{}'{D X} \iso \Inv T_{X} , \quad \xi \otimes x \mapsto \xi * \kappa(x)\quad 
(\xi\in \Inv X, \, x\in {}'{\Sym(X^*)}).
\]
\end{Theorem}
\begin{proof}
The map $\phi$ in the theorem is an isomorphism of $\O$-supermodules by Lemmas~\ref{LemmaTri} and \ref{L120116}. To see that it is an algebra homomorphism, we compute for 
$\xi,\eta \in \Inv X$ and $x,y\in {}'{\Sym(X^*)}$:
\begin{align*}
\phi\big((\xi\otimes x)(\eta\otimes y)\big)&=\phi\big(\sum (-1)^{\bar\xi_{(1)}(\bar\xi_{(2)}+\bar\eta+\bar x)+\bar \eta_{(1)}\bar x} 
 \xi_{(2)}\eta_{(1)}\otimes (x\cdot\eta_{(2)})(\xi_{(1)}\cdot y)\big)
\\
&=\sum (-1)^{\bar\xi_{(1)}(\bar\xi_{(2)}+\bar\eta+\bar x)+\bar \eta_{(1)}\bar x} 
 (\xi_{(2)}\eta_{(1)}) * \kappa\big((x\cdot\eta_{(2)})(\xi_{(1)}\cdot y)\big)
\\
&=\sum (-1)^{\bar\xi_{(1)}(\bar\xi_{(2)}+\bar\eta+\bar x)+\bar \eta_{(1)}\bar x} 
 (\xi_{(2)}\eta_{(1)}) * \kappa(x\cdot\eta_{(2)})*\kappa(\xi_{(1)}\cdot y)
 \\
&=\sum (-1)^{\bar\xi_{(1)}(\bar\xi_{(2)}+\bar\eta+\bar x)+\bar \eta_{(1)}\bar x} 
 (\xi_{(2)}\eta_{(1)}) * (\kappa(x)\cdot\eta_{(2)})*(\xi_{(1)}\cdot \kappa(y))
\\
&=(\xi*\kappa(x))(\eta*\kappa(y))
\\
&=
\phi(\xi\otimes x)\phi(\eta\otimes y),
\end{align*}
where we have used (\ref{EProductDoubleBasic}) for the first equality, Lemma~\ref{LemmaTri}  for the third equality,
Lemma~\ref{L140116_2} for the fourth equality
 and  Lemma~\ref{L140116_1} for the fifth equality. 
\end{proof}

\begin{Example} \label{ExTriv} 
{\rm 
Let $\O[z]_d$ be the truncated polynomial algebra $\O[z]/(z^{d+1})$, and ${}'{\O[z]_d}$ be the divided power truncated polynomial algebra defined as the  $\O$-subalgebra of $\K[z]/(z^{d+1})$ spanned by all $z^{(e)}$ with $e=0,\dots,d$.
If $X$ is the trivial algebra $\O$, let $y\in X^*$ be the function which sends $1$ to $1$. Then $D^d X\cong \O[z]_d$, with $1^{\otimes d-e}\otimes y^{e}\in \Inv^{d-e} X\otimes \Sym^e(X^*)$ corresponding to $z^e$, and ${}'{D^d X}\cong {}'{\O[z]_d}$, with $1^{\otimes d-e}\otimes y^{(e)}\in \Inv^{d-e} X\otimes {}'{\Sym^e(X^*)}$ corresponding to $z^{(e)}$.
}
\end{Example}

\subsection{A generating set for a Turner double}\label{SSGenSet}
For any $d\in\Z_{\ge 0}$, define
$\DC^d X\subseteq \Inv^d T_{X}$ to be the image of $D^d X$ under the isomorphism of Theorem~\ref{TFund}, and set 
$\DC X := \bigoplus_{d\ge 0} \DC^d X$. Of course $\DC^d X$ is just an isomorphic copy of $D^d X$, considered as an explicit subalgebra of $\Inv^d T_{X}$. 
By~\eqref{EDXZ} and Lemma~\ref{LemmaTri}, we have 
\begin{equation}\label{EDCdXZ}
\DC^d X = \bigoplus_{e=0}^d \Inv^{d-e} (X) * \Star^e (X^*). 
\end{equation}

\label{SSGen}
Let  $Y = X_{\1} \oplus X^*$, so that $Y$ is naturally an $X_{\0}$-superbimodule and 
$T_{X} = X_{\0} \oplus Y$.

\begin{Lemma}\label{LY}
For any $d\in \Z_{\ge 0}$, we have 
\[
\DC^d X = \bigoplus_{e=0}^d \Inv^{d-e} (X_{\0}) * \Star^e Y.
\]
\end{Lemma}

\begin{proof}
 By Lemma~\ref{L120116}, 
\[
\Inv^{d-e} (X) = \bigoplus_{f=0}^{d-e} \Inv^{d-e-f} (X_{\0}) * \Inv^f (X_{\1}).
\]
It follows from Lemma~\ref{LemmaTri} that $\Inv^f (X_{\1})=\Star^f (X_{\1})$ for all $f\in \Z_{\ge 0}$, so by~\eqref{EDCdXZ} we have 
\begin{align*}
\DC^d X &= \bigoplus_{e=0}^d \bigoplus_{f=0}^e \Inv^{d-e-f} (X_{\0}) *
\Star^f (X_{\1}) * \Star^{e} (X^*)  \\
&=
\bigoplus_{e=0}^d \Inv^{d-e} (X_{\0}) * \Star^e Y. \qedhere
\end{align*}
\end{proof}

In the rest of this subsection, we write $1$ for the identity element $1_{X}$ of $X$. 

\begin{Theorem}\label{Lgen}
For any $d\in\Z_{>0}$,
the $\O$-superalgebra $\DC^d X$ 
is generated by $\Inv^d X_\0$ and $1^{\otimes (d-1)} * Y$. 
\end{Theorem}

\begin{proof}
Let $\GC$ be the subalgebra of $\DC^d X$ generated by 
$\Inv^d X_{ \0}$ and $1^{\otimes (d-1)} * Y$.  
By Lemma~\ref{LY}, it suffices to show that
$\DC^{d-e,e} X:=\Inv^{d-e} (X_{\0}) * \Star^e Y \subseteq \GC$ for all $e\in [0,d]$. 
We will prove this by induction on $e$, the case $e=0$ being clear.

Let $0<e\le d$ and assume that $\DC^{d-f,f} X \subseteq \GC$ 
for all $f\in[0,e)$. 
Let $x\in Y$ and $y\in \Star^{e-1} Y$. 
It follows from Lemma~\ref{Lstfo} that 
\[
(1^{\otimes (d-1)}*x)(1^{\otimes (d-e+1)} * y) \in 1^{\otimes (d-e)} * x *y + 
\bigoplus_{f=0}^{e-1} \DC^{d-f,f} X.
\]
So $1^{\otimes (d-e)} * \Star^e Y \subseteq \GC$. 

For every $f\in [0,d-e]$, write 
$\DC^{d-e-f,f,e} X := 1^{\otimes (d-e-f)} * \Inv^f (X_{\0}) * \Star^e Y$.
We claim that $\DC^{d-e-f,f,e} X\subseteq \GC$ for all such $f$. If the claim is true, then $\DC^{d-e,e} X= \DC^{0,d-e,e} X \subseteq \GC$, which implies the lemma. 
We prove the claim by induction on $f$. 
The base case $f=0$ was established in the previous paragraph.

Given $f\in (0,d-e]$ and assuming that our claim is true for smaller $f$, let 
$\xi\in \Inv^{f} (X_{\0})$ and $z\in \Star^e Y$. By Lemma~\ref{Lstfo}, we have 
\begin{align*}
&(1^{\otimes (d-f)} * \xi) ( 1^{\otimes (d-e)} * z) = \\
& \quad =
\sum_{a=0}^{\min(d-e,d-f)} 
\sum \pm 
(1^{\otimes a}) *
(\xi_{(1)} 1^{\otimes (d-e-a)}) 
*(1^{\otimes (d-f-a)} z_{(1)})
* (\xi_{(2)} z_{(2)})
\\
& \quad =
\sum_{a=0}^{\min(d-e,d-f)} 
\sum \pm 
(1^{\otimes a}) *
(\xi_{(1)} 1^{\otimes (d-e-a)}) 
* (\xi_{(2)} z_{(2)})
*(1^{\otimes (d-f-a)} z_{(1)}),
\end{align*}
where supercommutativity of $*$ has been used for the last equality. Note that $\xi_{(1)} \in \Inv^b (X_{\0})$ for some $b\le f$, so $\xi_{(1)} 1^{\otimes (d-e-a)}=0$ if $a<d-e-f$. Moreover,  any term in the sum with $a>d-e-f$ belongs to $\DC^{a,d-e-a,e} X$ and hence to $\GC$ by the inductive hypothesis. The remaining term is $1^{\otimes (d-e-f)}* \xi * z$, so 
$1^{\otimes (d-e-f)} * \xi * z\in \GC$, and we have proved our claim.
\end{proof}

Let $W$ be an $X_\0$-bimodule.
For any $\xi\in X_\0$, define $\ad(\xi)\in \End_\O (W)$
by
$
\ad(\xi) (w) := \xi w - w\xi
$
for all $w\in W$. 
Further, for any $r\in \Z_{\ge 0}$, we define  
$\ad^r (X_{\0})\subseteq \End_\O (W)$ as the $\O$-span of all compositions
$\ad(\xi_1) \circ \dots \circ \ad(\xi_r)$ for $\xi_1,\dots,\xi_r\in X_\0$. 
As usual, if $F$ is a subset of $\End_\O (W)$ and $U$ is a subset of $W$, we denote by $F(U)$ the $\O$-span of the elements $f(u)$ for all $f\in F$ and $u\in U$. 

\begin{Corollary}\label{CGen}
Let $U$ be a subsuperspace of $Y$ such that 
$\sum_{r\ge 0} \ad^r (X_\0) (U) = Y$, and let $d\in \Z_{>0}$. Then  the $\O$-superalgebra $\DC^d X$ is generated by 
$\Inv^d (X_{\0})$ and $1^{\otimes (d-1)} * U$.
\end{Corollary}

\begin{proof}
If $d=1$, the result is clear, so we assume that $d\ge 2$. 
By Lemma~\ref{Lstfo}, for any $\xi\in X_{\0}$ and $x\in Y$, we have 
\begin{align*}
 (1^{\otimes (d-1)}* \xi)(1^{\otimes (d-1)} * x) &= 
 1^{\otimes (d-2)} * \xi * x+1^{\otimes (d-1)} * (\xi x), \\
 (1^{\otimes (d-1)}* x)(1^{\otimes (d-1)} * \xi) &= 
 1^{\otimes (d-2)} * x * \xi+1^{\otimes (d-1)} * (x \xi).
\end{align*}
Since $\xi$ has degree $\0$, we have $x * \xi=\xi* x$, so
\[
1^{\otimes (d-1)} * (\ad (\xi) (x)) = (1^{\otimes (d-1)}* \xi)(1^{\otimes (d-1)} * x)
- (1^{\otimes (d-1)}* x)(1^{\otimes (d-1)} * \xi).
\]
We have proved that if $1^{\otimes (d-1)}*x$ belongs to the subalgebra $\GC\subseteq \DC^d X$ 
generated by $\Inv^d (X_{\0})$ and $1^{\otimes (d-1)} * U$, then 
$1^{\otimes (d-1)}*(\ad(\xi) (x))\in \GC$ for all $\xi\in X_\0$. 
In view of the hypothesis, this implies that 
$1^{\otimes (d-1)} * Y\subseteq \GC$, and the result now follows by Theorem~\ref{Lgen}.
\end{proof}

\subsection{Gradings}\label{SSGr}
By (\ref{EProductDoubleBasic}), the algebra $D^d X$ (resp.~${}'{D^dX}$) is $\Z_{\geq 0}$-graded with the graded degree $e$ component being $D^{d-e,e}X$ (resp.~${}'{D^{d-e,e}X}$) for $e=0,\dots,d$.  We refer to this grading as the {\em standard grading}. In fact, it is a {\em supergalgebra grading}, which means that it is an algebra grading and a supergrading in the sense of Section~\ref{SSuper}. If a superalgebra has a superalgebra grading, we just say that it is graded. 

Assume now that the multiplication in $X$ satisfies $X_\1 X_\1=0$. Then $X$ is a $\Z$-graded algebra with $X^0=X_\0$, $X^1=X_\1$ and $X^m=0$ for $m\neq 0,1$. 
We will always work with the grading on $X^*$ which is the {\em shift by $2$} of the canonical grading, i.e.~$\deg \xi=2$ if $\xi\in X^*$ satisfies $\xi(X^1)=0$ and $\deg \xi=1$ if $\xi\in X^*$ satisfies $\xi(X^0)=0$. Now $T_X=X\oplus X^*$ is also graded, and it is easy to see that this is a superalgebra grading.

This yields $\Z_{\geq 0}$-gradings on 
$\Inv X$, $\Sym(X^*)$, ${}'{\Sym(X^*)}$ and $\Inv T_X$. Moreover, we let $(\Inv X)^*$ inherit the grading from $\Sym(X^*)$ via the identification (\ref{EId}). So we have $\Z_{\ge 0}$-gradings
 on the $\O$-superspaces 
$DX=\Inv X\otimes (\Inv X)^*$ and ${}'{DX}=\Inv X\otimes {}'{\Sym(X^*)}$, which we refer to as {\em Turner's gradings}, cf. \cite[Remark 156]{Turner}. 
If $Y=DX$ or ${}'{DX}$ with Turner's grading, then $Y_\0=\bigoplus_{m\ \text{even}}Y^m$ and $Y_\1=\bigoplus_{m\ \text{odd}}Y^m$. In particular, Turner's grading is a supergrading. 

\begin{Lemma} \label{LGrDNew} 
Let the superalgebra $X$ have the property that $X_\1 X_\1=0$. Then, for every $d\in\Z_{\geq 0}$, the superalgebras $D^dX$ and ${}'{D^dX}$ are $\Z_{\geq 0}$-graded with respect to Turner's gradings. Moreover, the isomorphism of Theorem~\ref{TFund} is an isomorphism of graded superalgebras. 
\end{Lemma}
\begin{proof}
It is easy to check that $\Inv X$, $\Sym(X^*)$, ${}'{\Sym(X^*)}$ are $\Z_{\geq 0}$-graded superalgebras. Moreover, $\Inv(X^*)$ is graded with respect to the $*$-product. Next, one checks that both $\Sym(X^*)$ and ${}'{\Sym(X^*)}$ are graded $\Inv X$-bimodules. Finally, the homomorphisms $\Delta\colon \Inv X\to \Inv X\otimes \Inv X$ and $\kappa\colon {}'{\Sym (X^*)} \iso \Inv(X^*)$ are  homogeneous of degree zero. 
So the lemma follows from (\ref{EProductDoubleBasic}). 
\end{proof}

\subsection{Symmetricity of doubles}
\label{SSDoubleSymm}
Let $X$ be a $\k$-superalgebra which is free of finite rank as a $\k$-module. The Turner double superalgebra $D^d X$ defined in \S\ref{SSDoubles} is {\em symmetric}. To see this, we define the bilinear form on $D^d X$ via
$$
(\xi\otimes x,\eta\otimes y):=\langle \xi,y\rangle\langle x,\eta\rangle.
$$

We give another description of the form $(\cdot,\cdot)$. 
Recall the standard grading on $D^d X$ from \S\ref{SSGr}.
Let $F\in (D^d X)^*$ be defined by requiring that $F$ is zero on all standard graded components $D^{d-e,e}X$ for $0\leq e<d$, and $F(1\otimes x)=x(1_X^{\otimes d})$ for $x\in (\Inv^d X)^*$. 

\begin{Lemma} \label{LFunction} 
For any $t,u\in D^d X$, we have $(t,u) = F(tu)$. 
\end{Lemma}
\begin{proof}
We may assume that $t=\xi\otimes x$ and $u=\eta\otimes y$, where
$\xi\in \Inv^{d-e} X$,
 $x\in (\Inv^e X)^*$, $\eta\in \Ind^{d-f} X$ and $y\in (\Inv^f X)^*$ for some $0\leq e,f\leq d$. We may further assume that $e=d-f$, for otherwise both sides of the equation in the lemma are zero. Then, using~\eqref{EProductDoubleBasic}, we have
\begin{align*}
F\big((\xi\otimes x)(\eta\otimes y)\big)&=
\sum (-1)^{\bar\xi_{(1)}(\bar\xi_{(2)}+\bar\eta+\bar x)+\bar \eta_{(1)}\bar x} 
F\big( \xi_{(2)}\eta_{(1)}\otimes (x\cdot\eta_{(2)})(\xi_{(1)}\cdot y)\big)
\\
&=(-1)^{\bar\xi(\bar\eta+\bar x)} 
F\big( 1\otimes (x\cdot\eta)(\xi\cdot y)\big)
\\
&=(-1)^{\bar\xi(\bar\eta+\bar x)}\big((x\cdot\eta)(\xi\cdot y)\big)(1_X^{\otimes d})
\\
&=(-1)^{\bar\xi(\bar\eta+\bar x)}\big((x\cdot\eta)(1_X^{\otimes e})\big)\big((\xi\cdot y)(1_X^{\otimes f})\big)
\\
&=(-1)^{\bar\xi(\bar\eta+\bar x)}\langle x,\eta\rangle\langle \xi,y\rangle,
\end{align*}
where we have used (\ref{EActions})  for the last equality. It remains to note that we can drop the sign since $\langle x,\eta\rangle=0$ unless $\bar x=\bar\eta$. 
\end{proof}

Note that over an arbitrary $\k$, non-degeneracy of a bilinear form $(\cdot,\cdot)$ on a free $\k$-module $V$ of a finite rank means that for every $\k$-basis $\{v_1,\dots,v_m\}$ of $V$ there is another basis $\{w_1,\dots,w_m\}$ such that $(v_a,w_b)=\de_{a,b}$. The following corollary shows that $D^d X$ is a {\em symmetric algebra}. 

\begin{Corollary}\label{CSymmetric} \cite[Theorem 1.1]{TurnerCat} 
The form $(\cdot,\cdot)$ on $D^dX$ is non-degenerate, symmetric and associative. 
\end{Corollary}
\begin{proof}
The non-degeneracy and symmetricity are clear, while the associativity follows from Lemma~\ref{LFunction}. 
\end{proof}

\section{Generalized Schur-Weyl duality}\label{SXSW}
Throughout this section, $A=A_\0\oplus A_\1$ is a $\k$-superalgebra with $\k$-bases $\ttB_\0$ of $A_\0$, $\ttB_\1$ of $A_\1$, and $\ttB=\ttB_\0\sqcup \ttB_\1$ of $A$. 
Fix $d\in\Z_{\geq 0}$ and $n\in \Z_{> 0}$.

\subsection{Wreath product superalgebras}
\label{SSWreath}
We will consider {\em super wreath products} 
\begin{equation}\label{EWreath}
W^A_d:= A^{\otimes d}\rtimes \k\Si_d,
\end{equation}
with $\k\Si_d$ concentrated in degree $\0$. 
We identify $A^{\otimes d}$ and $\k\Si_d$ with the subsuperalgebras $A^{\otimes d}\otimes 1_{\Si_d}$ and $1_{A}^{\otimes d}\otimes \k\Si_d$ of $W^A_d$, respectively. 
The multiplication in $W^A_d$ is then uniquely determined by the additional requirement that 
\begin{equation}\label{EWreathProductDetermined}
g^{-1}(x_1\otimes\dots\otimes x_d)g=(x_1\otimes\dots\otimes x_d)^g
\end{equation}
for $g\in\Si_d$ and $x_1,\dots,x_d\in A$, see (\ref{ESGAction}). 
Given $x\in A$ and $1\leq c\leq d$, we denote 
$$
x[c]:=1_A\otimes \dots\otimes1_A\otimes x\otimes 1_A\otimes\dots\otimes  1_A\in A^{\otimes d},
$$
with $x$ in the $c$th position. 
The following lemma is obvious:

\begin{Lemma} \label{LWrChar}
Let $A$ be a superalgebra and $d\in \Z_{\geq 0}$. 
Then the superalgebra $W^A_d$ 
is generated by the elements $\{x[c]\mid x\in A,\ 1\leq c\leq d\}\sqcup \Si_d$ subject only to the following relations:
\begin{align*}
x[c]\cdot y[c]&=xy[c]&(x,y\in A,\ 1\leq c\leq d),
\\
x[b]\cdot y[c]&=(-1)^{\bar x\bar y}y[c]\cdot x[b] &(x,y\in A,\ 1\leq b\neq c\leq d),
\\
g\cdot  h&=gh &(g,h\in \Si_d),
\\
 g \cdot x[c]&= x[gc]\cdot g &(g\in \Si_d,\ x\in A,\ 1\leq c\leq d).
\end{align*}
\end{Lemma}

Let $\la\in\La(n,d)$. We always consider the group algebra $\k\Si_\la$ of the standard parabolic subgroup $\Si_\la$ as a subalgebra $\k\Si_\la\subseteq \k\Si_d\subseteq W^A_d$. In particular, $\k\Si_\la$ acts naturally on the left on $W^A_d$. This makes $W^A_d$ into a left $\k\Si_\la$-module, which is free with basis 
$$\{g(\ttb_1\otimes\dots\otimes \ttb_d) \mid g\in  {}^\la\D,\ \ttb_1,\dots,\ttb_d\in \ttB\}.$$ 
So, denoting by $\triv_{\la}$ the trivial right $\k\Si_\la$-module 
$\k \cdot 1_\la$, 
we have the (right) induced   $W^A_d$-module 
\begin{equation}\label{EMMuNew}
M_\la^A:=\triv_{\la}\otimes_{\k\Si_\la} W^A_d
\end{equation}
with generator $m_\la:=1_\la\otimes 1$. We refer to $M_\la^A$ as a {\em permutation module}.

\subsection{Tensor space}\label{SSTensorSpace}
 The matrix algebra $M_n(A)$ is a superalgebra in its own right. We use the elements 
\begin{equation}\label{EXiX}
\xi_{r,s}^x:=xE_{r,s}\in M_n(A)\qquad (x\in A,\ 1\leq r,s\leq n)
\end{equation}
as in \S\ref{SSTEA}. 
We also introduce the special notation 
$$
S^A(n,d):=\Inv^d (M_n(A))\quad\text{and}\quad
S^A(n):=\Inv (M_n(A))=\bigoplus_{d\geq 0} S^A(n,d).
$$
If $A=\k$, the algebra $S^A(n,d)$ is nothing but the classical Schur algebra $S(n,d)$ as in \cite{Green}.

Let $V=A^{\oplus n}$, considered as a right $A$-supermodule in the natural way. Note that we have a natural isomorphism
$
M_n(A)\iso \End_A(V),
$ 
where we consider $V$ as column vectors and the isomorphism sends a matrix $\xi$ to the left multiplication by $\xi$. This implies the isomorphism
\begin{equation}\label{EIsoEnd}
\Tens^d M_n(A)\iso \End_{\Tens^d A}(\Tens^dV).
\end{equation}

Recall from (\ref{ESGAction}) that $\Si_d$ acts on $\Tens^dV$ with $\k$-linear maps, and write $vg:=v^g$ for $v\in V$, $g\in \Si_d$.
 Thus we have right supermodule structures  on $\Tens^dV$ over both $k\Si_d$ and $\Tens^d A$. 
 In view of Lemma~\ref{LWrChar}, the superspace $\Tens^dV$ becomes a right $W^A_d$-supermodule. We refer to this right action of $W^A_d$ on $\Tens^dV$ as the {\em standard permutation action}. 

 \begin{Lemma} \label{LSchurIso} 
 The natural embedding 
 $$S^A(n,d)\,\into\, \Tens^d M_n(A)\iso \End_{\Tens^d A}(\Tens^dV)$$ 
 defines an isomorphism of superalgebras 
 $$
S^A(n,d)\cong \End_{W^A_d}(\Tens^d V).
 $$
 \end{Lemma}
 \begin{proof}
The action of $\Si_d$  on $\Tens^dV$ yields the action on $\End_{\Tens^d A}(\Tens^dV)$ via $(\phi\cdot g)(v)=\phi (v g^{-1})g$
for $\phi\in \End_{\Tens^d A}(\Tens^dV)$, $g\in\Si_d$ and $v\in \Tens^dV$. 

Let $\al\colon \Tens^d M_n(A)\iso\End_{\Tens^d A}(\Tens^dV)$ be the isomorphism (\ref{EIsoEnd}). We have the $\Si_d$-action on $\End_{\Tens^d A}(\Tens^dV)$ defined in the previous paragraph, and the $\Si_d$-action on $\Tens^d M_n(A)$ defined by 
(\ref{ESGAction}). 
It is easy to see that $\al$ intertwines the two actions. 
Taking invariants, 
we get an isomorphism between $S^A(n,d)=(\Tens^d M_n(A))^{\Si_d}$ and 
$\End_{W^A_d}(\Tens^d V)=(\End_{\Tens^d A}(\Tens^d V))^{\Si_d}.
$
 \end{proof}

For $1\leq r\leq n$, we set 
\begin{equation}\label{EVR}
v_r:=(0,\dots,0,1_A,0,\dots,0)\in V,
\end{equation}
where $1_A$ is in the $r$th position. For $\br=(r_1,\dots,r_d)\in\Seq(n,d)$, we define
$$
v_\br:=v_{r_1}\otimes\dots\otimes v_{r_d}\in \Tens^d V.
$$
Since $\{v_1,\dots,v_n\}$ is an $A$-basis of $V$, the set 
$
\{v_\br\mid \br\in \Seq(n,d)\}
$
is a $\Tens^d A$-basis of $\Tens^d V$. Note that 
\begin{equation}\label{EEBRG}
v_\br g= v_{\br g}\qquad(g\in\Si_d,\ \br\in \Seq(n,d)).
\end{equation}
Let $\la\in\La(n,d)$. 
We denote by $\Tens^\la V$ the $\Tens^d A$-span of all $v_\br$ such that $\br\in{}^\la\Seq$, cf.~(\ref{ELaSeq}). By (\ref{EEBRG}), $\Tens^\la V$ is a $W^A_d$-submodule of $\Tens^d V$. 
We have a special vector 
$$
v_\la:=v_1^{\otimes \la_1}\otimes\dots\otimes v_n^{\otimes \la_n}\in \Tens^\la V.
$$ 
We have the decomposition of $W^A_d$-modules: 
\begin{equation}\label{E100216}
\Tens^d V=\bigoplus_{\la\in\La(n,d)} \Tens^\la V.
\end{equation}

\begin{Lemma} \label{LTLaMLa} 
Let $\la\in\La(n,d)$. There is an isomorphism of  right $W_d^A$-modules 
$
\Tens^\la V\iso M_\la^A
$
which maps $v_\la$ to the standard generator $m_\la$ of~$M_\la^A$.
\end{Lemma}
\begin{proof}
It is immediate that $v_\la$ is $\Si_\la$-invariant, which yields a homomorphism 
$
M_\la^A\to \Tens^\la V,\ m_\la\mapsto v_\la. 
$ 
This is an isomorphism, since it maps the $\Tens^d A$-basis $\{m_\la g\mid g\in{}^\la\D\}$ of $M_\la^A$ to the $\Tens^d A$-basis 
$\{v_\br\mid \br\in {}^\la\Seq\}$ of $\Tens^\la V$, cf. the bijection (\ref{ESingleCosetSeq}). 
\end{proof}

For any $\la\in\La(n,d)$, we define
\begin{equation}\label{EXiLa}
\xi_\la:=E_{1,1}^{\otimes \la_1}*\dots* E_{n,n}^{\otimes \la_n}\in S^A(n,d).
\end{equation}

\begin{Lemma} \label{L189116} 
Let $\la,\mu\in\La(n,d)$. Then:
\begin{enumerate}
\item[{\rm (i)}] $\xi_\la\xi_\mu=\de_{\la,\mu}\xi_\la$ and\, 
$\sum_{\nu\in \La(n,d)}\xi_\nu=1$.
\item[{\rm (ii)}] $\xi_\la \Tens^d V=\Tens^\la V$. 
\end{enumerate}
\end{Lemma}
\begin{proof}
Note that $\xi_\la v_\mu=\de_{\la,\mu}v_\la$. But $v_\la$ generates $\Tens^\la V$ as a right $W^A_d$-module by Lemma~\ref{LTLaMLa}, and the action of $S^A(n,d)$ on $\Tens^d V$ commutes with that of $W^A_d$ by Lemma~\ref{LSchurIso}, so $\xi_\la$ acts as the projection onto $\Tens^\la V$ along $\bigoplus_{\nu\neq \la}\Tens^\nu V$. The lemma follows since $S^A(n,d)$ acts on $\Tens^d V$ faithfully thanks to Lemma~\ref{LSchurIso}.  
\end{proof}

\subsection{Idempotent truncation} \label{SSIdempotentTruncation} 
Throughout the subsection we assume that $d\leq n$ and set 
\begin{equation}\label{EOm}
\om:=\eps_1+\dots +\eps_d\in\La(n,d). 
\end{equation}
The main goal of this subsection is to explicitly identify $\xi_\om S^A(n,d)\xi_\om$ with $W_d^A$ and $S^A(n,d)\xi_\om$ with $\Tens^d V$ so that the natural right action of $\xi_\om S^A(n,d)\xi_\om$ on $S^A(n,d)\xi_\om$ becomes the  standard permutation action of $W_d^A$ on $\Tens^d V$, cf.~\cite[Chapter 6]{Green} for the case when $A=\k$.

\begin{Lemma} \label{LIdempotentTruncation}
There is a superalgebra isomorphism 
$$
\phi\colon W^A_d\iso \xi_\om S^A(n,d)\xi_\om,\ (x_1\otimes\dots\otimes x_d)g\mapsto \xi^{x_1}_{1,g^{-1}1}*\dots*\xi^{x_d}_{d,g^{-1}d}.
$$
Moreover, for any $w\in W^A_d$, its image $\phi(w)$ is the unique element of $\xi_\om S^A(n,d)\xi_\om$ such that $\phi(w)v_\om=v_\om w$. 
\end{Lemma}
\begin{proof}
Using Lemma~\ref{LSchurIso}, we have an isomorphism of superalgebras 
\begin{align*}
\al\colon \xi_\om S^A(n,d)\xi_\om\iso \End_{W_d^A}(\xi_\om \Tens^d V)
\end{align*}
which maps $s\in \xi_\om S^A(n,d)\xi_\om$ to the left multiplication by $s$. On the other hand, by Lemma~\ref{LTLaMLa}, there is an isomorphism of right $W_d^A$-supermodules $\xi_\om \Tens^d V=\Tens^\om V\iso M_{\om}^A,\ v_\om\mapsto m_\om$. But the $W_d^A$-module $M_{\om}^A$ is free of rank $1$ with generator $m_\om$. So there is an  isomorphism  
$\be\colon W_d^A\iso \End_{W_d^A}(\xi_\om \Tens^d V)$ of superalgebras which sends $w\in  W_d^A$ to the endomorphism $v_\om\mapsto v_\om w$. 

Generalizing the notation (\ref{EVR}), we set 
\begin{equation}\label{EVGen}
v_r^x:=(0,\dots,0,x,0,\dots,0)\in V \qquad(1\leq r\leq n,\ x\in A),
\end{equation}
where $x$ is in the $r$th position. Then 
\begin{align*}
\al(\xi^{x_1}_{1,g^{-1}1}*\dots*\xi^{x_d}_{d,g^{-1}d})(v_\om)
&=
(\xi^{x_1}_{1,g^{-1}1}*\dots*\xi^{x_d}_{d,g^{-1}d})(v_1\otimes\dots\otimes v_d)
\\
&=(-1)^{[g;x_1,\dots,x_d]} v_{g1}^{x_{g1}}\otimes\dots\otimes v_{gd}^{x_{gd}}
\\&=(v_{1}^{x_{1}}\otimes\dots\otimes v_{d}^{x_{d}})g
\\&=v_\om(x_1\otimes \dots\otimes x_d)g
\\&=\be((x_1\otimes \dots\otimes x_d)g)(v_\om).
\end{align*}
This proves the lemma.
\end{proof}

Note that $S^A(n,d)\xi_\om$ is a right $\xi_\om S^A(n,d)\xi_\om$-module, so we consider it as a right $W_d^A$-module via the identification of $W_d^A$ with $\xi_\om S^A(n,d)\xi_\om$ coming from the isomorphism $\phi$ of Lemma~\ref{LIdempotentTruncation}.

\begin{Proposition} \label{PSWFund} 
There is a unique isomorphism $S^A(n,d)\xi_\om\iso \Tens^d V$ of $(S^A(n,d),W^A_d)$-superbimodules which maps 
$\xi_\om$ to $v_\om.$
\end{Proposition}
\begin{proof}
Since $\xi_\om v_\om=v_\om$, there is a unique homomorphism $\psi$ of left $S^A(n,d)$-modules 
$S^A(n,d)\xi_\om\to \Tens^d V$ mapping  $\xi_\om$ to $v_\om$. 
Using Lemma~\ref{LIdempotentTruncation}, we compute for any $s\in S^A(n,d)$ and any $w\in W^A_d$:
\begin{align*}
\psi((s\xi_\om) w)&=\psi((s\xi_\om)\phi(w)) = \psi(s\xi_\om\phi(w)) 
=\psi(s\phi(w)\xi_\om) 
\\
&=s\phi(w)v_\om = sv_\om w =s\xi_\om v_\om w=\psi(s\xi_\om)w,
\end{align*}
so $\psi$ is a homomorphism of $(S^A(n,d),W^A_d)$-superbimodules. 

Moreover, $\psi$ is injective since $\psi(s\xi_\om)=0$ only if $s\xi_\om v_\om=0$, which implies that $s\xi_\om \Tens^\om V=0$ because $v_\om W_d^A=\Tens^\om V$.
On the other hand, by Lemma~\ref{L189116}(ii), we have $s\xi_\om \Tens^\mu V=0$ for all $\mu\neq \om$, hence $s\xi_\om \Tens^d V=0$, so $s\xi_\om=0$. 

Finally,  for every $\mu\in\La(n,d)$ there is a homomorphism of right $W^A_d$-modules $M^A_\om\to M^A_\mu,\ m_\om\mapsto m_\mu$, and so there is a homomorphism of right $W^A_d$-modules $\Tens^\om V\to \Tens^\mu V,\ v_\om\mapsto v_\mu$, see Lemma~\ref{LTLaMLa}. 
By Lemma~\ref{LSchurIso}, there is $s\in S^A(n,d)$ with $sv_\om=v_\mu$.  
As $v_\mu$ generates $\Tens^\mu V$ as a $W^A_d$-module for every $\mu\in\La(n,d)$, we now deduce that $v_\om$ generates $\Tens^d V$ as an $(S^A(n,d),W^A_d)$-bimodule. Hence $\psi$ is surjective. 
\end{proof}

Denote the center of an algebra $Y$ by $Z(Y)$. 
Recall from Section~\ref{SSuper} the notation $|X|$ for a superalgebra $X$.
The following technical result, in which we forget the superstructures, will be needed in \S\ref{SSSymLat}:

\begin{Lemma}\label{LInvertible}
Let $d\leq n$. If $z\in Z(|S^A (n,d)|)$ and  
$\xi_{\om} \in |S^A (n,d)| z$, then $z$ is invertible. 
\end{Lemma}

\begin{proof}
Let $S:=|S^A (n,d)|$ and $W:=|W^{A}_d|$. 
First, note that $\xi_{\la} z \xi_{\mu} = z \xi_{\la} \xi_{\mu} =0$ 
for any distinct $\la,\mu\in \La(n,d)$. So $z= \sum_{\la\in \La(n,d)} z_{\la}$, where $z_{\la} := z \xi_{\la}= \xi_{\la} z$. 

Let $\la\in \La(n,d)$. There is a unique $W$-module homomorphism sending $m_{\om}$ to $m_{\la}$, so 
by Lemmas~\ref{LSchurIso} and~\ref{LTLaMLa}, there exists a unique element 
$\xi_{\la,\om} \in \xi_{\la} S\xi_{\om}$ 
such that $\xi_{\la,\om} v_{\om} = v_{\la}$.

 
 By the hypothesis, there exists $y_{\om} \in S$ such that $y_{\om} z = \xi_{\om}$. Replacing $y_{\om}$ with $\xi_{\om} y_{\om} \xi_{\om}$, we may (and do) assume that $y_{\om} \in \xi_{\om} S \xi_{\om}$, and then it is easy to see that $y_{\om} \in Z(\xi_{\om} S \xi_{\om})$. 
 Let $\tilde{y}_{\om}\in W$ be the image of $y_{\om}$ under the isomorphism $\xi_{\om} S \xi_{\om} \iso W$ of 
Lemma~\ref{LIdempotentTruncation}. Then 
$\tilde{y}_{\om} \in Z( W)$ and 
$y_{\om} v_{\om} = v_{\om} \tilde{y}_{\om}$.
For any $g\in \Si_\la$, we have 
$
m_{\la} \tilde{y}_{\om} g = m_{\la} g \tilde{y}_{\om} = 
m_{\la} \tilde{y}_{\om}
$. 
Hence, there is a right $W$-module endomorphism of 
$M_\la$ sending $m_{\la}$ to $m_{\la} \tilde{y}_{\om}$. 
By Lemmas~\ref{LSchurIso} and~\ref{LTLaMLa}, this implies that 
there exists $y_{\la} \in \xi_{\la} S \xi_{\la}$ such that 
$y_{\la} v_{\la} = v_{\la} \tilde{y}_{\om}$. 
 Therefore,
\begin{align*}
z_{\la} y_{\la} v_{\la} &= z_{\la} v_{\la} \tilde{y}_\om = 
z_{\la} \xi_{\la,\om} v_{\om} \tilde{y}_\om =
z_{\la} \xi_{\la,\om} y_{\om} v_{\om} 
=z \xi_{\la,\om} y_{\om} v_{\om} 
 = \xi_{\la, \om} z y_{\om}v_{\om}  \\
 &= \xi_{\la,\om} \xi_{\om}v_{\om}
 =  \xi_{\la,\om}v_{\om} = v_{\la}.
\end{align*}
By Lemma~\ref{LSchurIso}, it follows that $z_{\la} y_{\la} = \xi_{\la}$. Setting 
$y:=\sum_{\la\in\La(n,d)} y_{\la}$, we have $zy=1$.
\end{proof}

\subsection{Idempotent refinements}\label{SSColored}
In this subsection we suppose that we are given a fixed finite family $\{e_1, \dots,e_l\}$ of non-zero orthogonal idempotents in $A$ with $\sum_{i=1}^l e_i=1_A$. 
Moreover, we assume that every $e_iA$ is free  as a $\k$-supermodule with a (homogeneous) finite basis ${}_i\ttB$, so that $\ttB=\bigsqcup_{i=1}^l \,{}_i\ttB$ is a $\k$-basis of $A$. 

Set $I:=\{1,\dots,l\}$. We order $[1,n]\times I$ lexicographically
and, recalling the theory of \S\ref{SSCosets}, consider the set of  compositions 
$
\La([1,n]\times I,d).
$
Given $\ula\in \La([1,n]\times I,d)$, we denote $\la_r^{(i)}:=\ula_{(r,i)}$ for $(r,i)\in [1,n]\times I$. We have the map
$$
\pi\colon \La([1,n]\times I,d)\to\La(n,d),\ \ula\mapsto (\textstyle \sum_{i\in I}\la^{(i)}_1,\sum_{i\in I}\la^{(i)}_2,\dots,\sum_{i\in I}\la^{(i)}_n).
$$

Let $\ula\in\La([1,n]\times I,d)$. We have the idempotent 
\begin{equation}\label{EIdempTensor}
e_{\ula}^A:=
e_{1}^{\otimes \la_1^{(1)}}\otimes\dots\otimes 
e_{l}^{\otimes \la_1^{(l)}}\otimes\dots\otimes 
e_{1}^{\otimes \la_n^{(1)}}\otimes\dots\otimes 
e_{l}^{\otimes \la_n^{(l)}}\in \Tens^d A. 
\end{equation}
Recalling the notation of \S\ref{SSCosets},  note that 
\begin{equation}\label{EBMu}
\ttB_{\ula}^A:=\{\ttb_1\otimes\dots\otimes \ttb_d\mid \ttb_a\in {}_{i}\ttB\ \text{if $a\in \Om^\ula_{(r,i)}$}\}
\end{equation}
is a $\k$-basis of $e^A_\ula \Tens^d A$. 
We define the {\em parabolic subalgebra} 
$$
W_{\ula}^A:=e^A_{\ula} \otimes \k\Si_{\ula}\subseteq W_d^A.
$$

Note that $e^A_{\ula}$ is the identity in $W_{\ula}^A$, and  $W_{\ula}^A$ is a (usually {\em non-unital}) subsuperalgebra in $W_d^A$, isomorphic to the group algebra $\k \Si_\ula$. So we may consider the {\em trivial}\, right supermodule $\triv^A_{\ula}=\k\cdot 1^A_{\ula}$ over $W_{\ula}^A$ with the action on the basis element $1^A_{\ula}$ given by $1^A_{\ula}\cdot(e^A_{\ula}\otimes g)=1^A_{\ula}$ for any $g\in\Si_\ula$. 

As usual, we view $\Tens^d A$ and $\k\Si_d$ as subsuperalgebras of $W_d^A$, so we can also view $e^A_{\ula}$ as an element of $W_d^A$. Then $e^A_{\ula}W_d^A$ is naturally a left $W_{\ula}^A$-module.  We now define the {\em colored permutation supermodule}
$$
M_{\ula}^A:=\triv^A_{\ula}\otimes_{W_{\ula}^A}e^A_{\ula}W_d^A
$$
with generator $m_{\ula}^A:=1^A_{\ula}\otimes e^A_{\ula}$.

\begin{Lemma} 
The following set is a $\k$-basis of $M_{\ula}^A$:
\begin{equation}\label{EBUMu}
\{m_{\ula}^A (\ttb_1\otimes \dots\otimes \ttb_d)g\mid \ttb_1\otimes\dots\otimes \ttb_d\in \ttB_{\ula}^A,\ g\in{}^\ula\D\}.
\end{equation}
In particular, $\dim M_{\ula}^A=|\Si_d:\Si_\ula|\prod_{i\in I}(\dim e_i A)^{\sum_{r=1}^n\la_r^{(i)}}$. 
\end{Lemma}
\begin{proof}
Note that  $$e^A_{\ula} W_d^A=
(e_{1}A)^{\otimes \la_1^{(1)}}\otimes\dots\otimes 
(e_{l}A)^{\otimes \la_1^{(l)}}\otimes\dots\otimes 
(e_{1}A)^{\otimes \la_n^{(1)}}\otimes\dots\otimes 
(e_{l}A)^{\otimes \la_n^{(l)}}\otimes \k\Si_d,$$ 
and so 
$$
\{e^A_{\ula} (\ttb_1\otimes \dots\otimes \ttb_d)g\mid \ttb_1\otimes\dots\otimes \ttb_d\in \ttB_{\ula}^A,\ g\in{}^\ula\D\}.
$$
is a basis of $e^A_{\ula}W_d^A$ as a left $W_{\ula}^A$-module. The lemma follows.
\end{proof}


Recalling (\ref{EVGen}), we define
$$
v_{r,i}:=v_r^{e_i}=(0,\dots,0,e_i,0,\dots,0)\in V \qquad(1\leq r\leq n,\ i\in I),
$$
where $e_i$ is in the $r$th position. For $\ula\in\La([1,n]\times I,d)$, we denote by $\Tens^\ula V\subseteq \Tens^d V$ the (right) $\Tens^d A$-span of all $v_{r_1,i_1}\otimes\dots\otimes v_{r_d,i_d}$ such that for every $(r,i)\in [1,n]\times I$ we have $\sharp\{a\in [1,d]\mid  (r_a,i_a)=(r,i)\}=\la^{(i)}_r$. We say that a sequence $((r_1,\ttb_1),\dots,(r_d,\ttb_d))$ of elements of $[1,n]\times \ttB$ is {\em of type $\ula$} if $\sharp\{a\in [1,d]\mid  r_a=r\  \text{and $\ttb_a\in {}_i\ttB$})\}=\la^{(i)}_r$. 
It is easy to see  that 
\begin{equation}\label{EBTensUMu}
\{v_{r_1}^{\ttb_1}\otimes \dots\otimes v_{r_d}^{\ttb_d}\mid 
\text{$((r_1,\ttb_1),\dots,(r_d,\ttb_d))$  
is  of type $\ula$}\}
\end{equation}
is a $\k$-basis of $\Tens^\ula V$. Hence for any $\la\in\La(n,d)$, we have a decomposition 
\begin{equation}\label{E290116}
\Tens^\la V=\bigoplus_{\ula\in\pi^{-1}(\la)}\Tens^\ula V
\end{equation}
of $\k$-modules. We have a special vector
$$
v_\ula:=v_{1,1}^{\otimes \la^{(1)}_1}\otimes\dots \otimes v_{1,l}^{\otimes \la^{(l)}_1}\otimes\dots\otimes
v_{n,1}^{\otimes \la^{(1)}_n}
\otimes\dots \otimes v_{n,l}^{\otimes \la^{(l)}_n}
\in\Tens^\ula V.
$$

\begin{Lemma} \label{L100216} 
We have:
\begin{enumerate}
\item[{\rm (i)}] For any $\ula\in \La([1,n]\times I,d)$, we have that  $\Tens^\ula V$ is a submodule of the right $W_d^A$-module $\Tens^d V$. Moreover, there is an isomorphism of right $W_d^A$-modules $\Tens^\ula V\iso M^A_\ula$ which maps $v_\ula$ to $m_\ula^A$. 
\item[{\rm (ii)}] For any $\la\in\La(n,d)$, we have $\Tens^\la V=\bigoplus_{\ula\in\pi^{-1}(\la)}\Tens^\ula V$ as right $W_d^A$-modules. In particular, $\Tens^d V =\bigoplus_{\ula\in\La([1,n]\times I,d)}\Tens^\ula V$ and $M^A_\la\cong \bigoplus_{\ula\in\pi^{-1}(\la)} M^A_\ula$ as right $W_d^A$-modules. 
\end{enumerate}
\end{Lemma}
\begin{proof}
Note that $v_{\ula}e^A_{\ula}=v_{\ula}$ and $v_{\ula}g= v_\ula$ for any $g\in \Si_{\ula}$. 
So, by the adjointness of induction and restriction, there exists a homomorphism of right $W_d^A$-modules $M^A_\ula\to \Tens^d V$
under which 
$m_\ula^A$ is mapped to $v_\ula$. It maps the elements of the $\k$-basis (\ref{EBUMu}) of $M_\ula^A$ to the elements of the $\k$-basis (\ref{EBTensUMu}) of $\Tens^\ula V$ up to signs. This proves (i). Part (ii) follows from (i),  (\ref{E290116}), (\ref{E100216}) and Lemma~\ref{LTLaMLa}. 
\end{proof}

By~Lemma~\ref{LSchurIso}, the superalgebra $S^A(n,d)$ acts naturally on $\Tens^d V$ with $W_d^A$-homomorphisms. 
But by Lemma~\ref{L100216}, we have an explicit identification of right $W_d^A$-modules 
$$
\Tens^d V =\bigoplus_{\ula\in\La([1,n]\times I,d)}\Tens^\ula V=\bigoplus_{\ula\in\La([1,n]\times I,d)} M^A_\ula.
$$ 
So, for any $y\in S^A(n,d)$, the endomorphism $v\mapsto yv$ of $\Tens^d V$ becomes identified with an endomorphism which we denote by $\phi(y)$ of $\bigoplus_{\ula\in\La([1,n]\times I,d)} M^A_\ula$. Recalling Lemma~\ref{LSchurIso} again, we deduce:

\begin{Corollary} \label{C100116} 
Let $M^A(n,d):=\bigoplus_{\ula\in \La([1,n]\times I,d)}M_\ula^A$. 
Then $\phi\colon S^A(n,d)\to\End_{W_d^A}(M^A(n,d))$ is a superalgebra isomorphism. 
\end{Corollary}

\subsection{Desuperization}\label{SSDesup} 
Recall 
from Section~\ref{SSuper}
that $|X|$ denotes the algebra obtained from a $\k$-superalgebra $X$ by forgetting the $\Z_2$-grading.
In particular, we have the associative algebra $|A|$ and the usual wreath product $W_d^{|A|}$, where the symmetric group acts on $|A|^{\otimes d}$ by place permutations without signs. On the other hand, we can consider the associative algebra $|W_d^A|$. In general, the algebras $W_d^{|A|}$and $|W_d^A|$ are not isomorphic. However, we describe one important situation when they are. 

Let $e^\0$ and $e^\1$ be orthogonal idempotents in $A$ with 
$1:=1_A=e^\0+e^\1$.
 We call such a pair of idempotents {\em adapted} if 
$A_\0=e^\0Ae^\0\oplus e^\1Ae^\1$ and $A_\1=e^\0Ae^\1\oplus e^\1Ae^\0$. 
Let $1\leq r<d$.  We denote  the elementary transposition $(r,r+1)\in\Si_d$ by $\tau_r$. If in addition $\eps_1,\eps_2\in \Z_2$,  we set 
$$
e^{\eps_1,\eps_2}[r]:=e^{\eps_1}[r]e^{\eps_2}[r+1]=1^{\otimes r-1}\otimes e^{\eps_1}\otimes e^{\eps_2}\otimes 1^{\otimes d-r-1}\in A^{\otimes d}.
$$

\begin{Lemma} \label{LDesup} 
Let $(e^\0,e^\1)$ be an adapted pair of idempotents in $A$. 
Then there is an isomorphism of associative $\k$-algebras 
\begin{align*}
\si\colon W_d^{|A|}&\iso |W_d^A|,
\\ 
x[t]&\mapsto \sum_{\eps_1,\dots,\eps_{t-1}\in\Z_2}(-1)^{(\eps_1+\dots+\eps_{t-1})\bar x}e^{\eps_1}\otimes\dots\otimes e^{\eps_{t-1}}\otimes x\otimes 1^{\otimes d-t},
\\
\tau_r&\mapsto \tau_r(e^{\0,\0}[r]+e^{\0,\1}[r]+e^{\1,\0}[r]-e^{\1,\1}[r]). 
\end{align*}
\end{Lemma}
\begin{proof}
It is straightforward to check for all admissible $r,t,x,y$ that the elements $\si(\tau_1),\dots,\si(\tau_{d-1})$ satisfy the Coxeter relations, that $\si(x[t])\si(y[t])=\si(xy[t])$, and that $\si(\tau_r)\si(x[t])=\si(x[t])\si(\tau_r)$ if $t\neq r,r+1$. 

Let $1\leq s< t\leq d$. Then $\si(x[t])\si(y[s])$ equals
\begin{align*}
\sum_{\eps_1,\dots,\eps_{t-1}\in\Z_2}(-1)^{p}e^{\eps_1}\otimes\dots\otimes 
e^{\eps_{s-1}}\otimes e^{\eps_s}y\otimes e^{\eps_{s-1}}\otimes\dots\otimes 
e^{\eps_{t-1}}\otimes x\otimes 1^{\otimes d-t},
\end{align*}
where 
$$p=(\eps_1+\dots+\eps_{t-1})\bar x+(\eps_1+\dots+\eps_{s-1})\bar y+\bar x\bar y,$$
and 
$\si(y[s])\si(x[t])$ equals
\begin{align*}
\sum_{\eps_1,\dots,\eps_{t-1}\in\Z_2}(-1)^{q}e^{\eps_1}\otimes\dots\otimes 
e^{\eps_{s-1}}\otimes ye^{\eps_s}\otimes e^{\eps_{s-1}}\otimes\dots\otimes 
e^{\eps_{t-1}}\otimes x\otimes 1^{\otimes d-t},
\end{align*}
where 
$$q=(\eps_1+\dots+\eps_{t-1})\bar x+(\eps_1+\dots+\eps_{s-1})\bar y.$$
Considering the $e^{\eps}ye^{\eps'}$ components in the $s$th tensor position for all $\eps,\eps'\in\Z_2$ in the expressions above, and taking into account that $e^{\eps}ye^{\eps'}=0$ unless $\bar y=\eps+\eps'$ since $(e^\0,e^\1)$ is adapted, we see that $\si(x[t])\si(y[s])=\si(y[s])\si(x[t])$. 

Let $x\in A$ and $1\le r<d$. Then, writing 
\begin{align*}
u:=\sum_{\eps_1,\dots,\eps_{r-1}\in \Z_2} 
(-1)^{(\eps_1+\cdots +\eps_{r-1}) \bar x} e^{\eps_1} \otimes \cdots \otimes e^{\eps_{r-1}}, \quad 
v:= \underbrace{1\otimes \cdots \otimes 1}_{d-r-1 \text{ times}},
\end{align*}
we have 
\begin{align*}
\si(x[r+1]) \si(\tau_r) &= \\
&\hspace{-1cm}=
(u \otimes (e^\0+(-1)^{\bar x}e^\1) \otimes x \otimes v) (e^{\0,\0}[r]+e^{\0,\1}[r]+e^{\1,\0}[r]-e^{\1,\1}[r]) \tau_r  \\
&\hspace{-1cm}=
 (u\otimes (e^\0+e^\1) \otimes e^\0 x \otimes v + u \otimes (e^\0-e^\1) 
\otimes e^\1x \otimes v) \tau_r \\
&\hspace{-1cm}=
\tau_r (u\otimes e^{\0}x\otimes 1 \otimes v
+ u\otimes e^{\1} x \otimes (e^{\0}-e^\1) \otimes v) \\
&\hspace{-1cm}=
\si(\tau_r) \si(x[r]),
\end{align*}
where the second equality is proved by a case-by-case check using the adaptedness of $(e^\0,e^\1)$.  
Since $\si(\tau_r)^2=1$, it follows also that 
$\si(\tau_r) \si(x[r+1])= \si(x[r]) \si(\tau_r)$. 

In view of Lemma~\ref{LWrChar}, we have an algebra homomorphism $\si$ as in the statement of the lemma. Moreover, it is easy to see that for each $g\in \Si_d$, the map
$\si$ restricts to an automorphism of the 
$\k$-submodule $A^{\otimes d} \otimes g$,  whence $\si$ is an isomorphism. 
\end{proof}

Let again $(e^\0,e^\1)$ be an adapted pair of idempotents in $A$. Assume in addition that we are given two finite families of non-zero orthogonal idempotents $\{e_i\mid i\in I^\0\}$ and 
$\{e_i\mid i\in I^\1\}$ such that 
$e^\0=\sum_{i\in I^\0} e_i$ and $e^\1=\sum_{i\in I^\1}e_i$. 
Let $I=I^\0\sqcup I^\1$ and recall the theory of \S\ref{SSColored}. In particular, $I$ is identified with $\{1,\dots,l\}$ for some $l$ and for any $\ula\in\La([1,n]\times I,d)$, we have the colored permutation supermodule $M_\ula^A$ over the superalgebra $W_d^A$. Forgetting the $\Z_2$-gradings, we get the $|W_d^A|$-module $|M_\ula^A|$. On the other hand, by Lemma~\ref{LDesup}, there is an isomorphism of algebras  $\si\colon W_d^{|A|}\iso |W_d^A|$. Composing with this isomorphism, we get the $W_d^{|A|}$-module $|M_\ula^A|^\si$. In other words, $|M_\ula^A|^\si=M_\ula^A$ as a $\k$-module, but the action is defined by $v h=v\si(h)$ for all $v\in M_\ula^A$ and $h\in W_d^{|A|}$.

For every $i\in I$ we define the sign $\zeta_i$ as follows:
$$
\zeta_i:=
\left\{
\begin{array}{ll}
+1 &\hbox{if $i\in I^\0$,}\\
-1 &\hbox{if $i\in I^\1$.}
\end{array}
\right.
$$
Recall the parabolic subgroup
$$
\Si_\ula=\Si_{\la^{(1)}_1}\times \dots\times \Si_{\la^{(l)}_1}\times\dots\times \Si_{\la^{(1)}_n}\times \dots\times \Si_{\la^{(l)}_n}\leq \Si_d.
$$
Let $\ell$ be the usual length function on a symmetric group, 
cf.~\S\ref{SSCosets}. 
Define the function 
$\eps_{\ula} \colon \Si_{\ula} \to \{ \pm 1\} \subseteq \k$ 
by 
\begin{equation}\label{Eeps}
\eps_{\ula} (g^{(1)}_1, \dots, g^{(l)}_1,\dots, g^{(1)}_n, \dots, g^{(l)}_n):=\zeta_{1}^{\ell(g^{(1)}_1)} \cdots \zeta_{l}^{\ell(g^{(l)}_1)}\dots \zeta_{1}^{\ell(g^{(1)}_n)} \cdots \zeta_{l}^{\ell(g^{(l)}_n)}
\end{equation}
for all $(g^{(1)}_1, \dots, g^{(l)}_1,\dots, g^{(1)}_n, \dots, g^{(l)}_n)\in \Si_\ula$.

The algebra $W_d^{|A|}$ has the parabolic subalgebra $W_{\ula}^{|A|}:=e_\ula^{|A|}\otimes \k\Si_\ula\cong \k\Si_\ula$ defined by analogy with the parabolic subsuperalgebra $W_\ula^{A}\subseteq W_d^A$. 
We define the {\em alternating}\, right module 
$\alt_{\ula}^{|A|}=\k\cdot 1_{\ula}^{|A|}$ 
over $W_{\ula}^{|A|}$ with the action on the basis element $1_{\ula}^{|A|}$ given by 
$$1_{\ula}^{|A|}\cdot(e_{\ula}^{|A|}\otimes g)=
\eps_{\ula} (g)
1_{\ula}^{|A|}  \qquad (g\in \Si_{\ula}).$$
As in the superalgebra situation,   $e_{\ula}^{|A|}W_d^{|A|}$ is naturally a left $W_{\ula}^{|A|}$-module.  We now define the {\em colored permutation module}
\begin{equation}\label{EMLaC}
M_{\ula}^{|A|}:=\alt_{\ula}^{|A|}\otimes_{W_{\ula}^{|A|}}e_{\ula}^{|A|}W_d^{|A|}
\end{equation}
with generator $m_{\ula}^{|A|}:=1_{\ula}^{|A|}\otimes e_{\ula}^{|A|}$.

\begin{Proposition} \label{PDesupM} 
There is an isomorphism of right $W_d^{|A|}$-modules 
$$
M_{\ula}^{|A|}\iso |M_\ula^A|^\si, \ m_\ula^{|A|}\mapsto m_\ula^A.
$$ 
\end{Proposition}
\begin{proof}
Let $\tau_r$ be an elementary transposition which belongs to 
$\Si_\ula$. This means that $r,r+1\in\Omega^\ula_{(s,i)}$ for some $(s,i)\in [1,n]\times I$. By Lemma~\ref{LDesup}, we have $\si(e_\ula^{|A|}\otimes \tau_r)=\zeta_i(e_\ula^{A}\otimes \tau_r)$. This implies that 
$m_\ula^A\si(e_{\ula}^{|A|}\otimes g)=\eps_{\ula} (g) m_\ula^A
$
for all $g\in\Si_\ula$. By adjointness of induction and restriction, we get a homomorphism of $W_d^{|A|}$-modules as in the statement. Since $|M_\ula^A|^\si$ is generated by $m_\ula^{A}$, this homomorphism is surjective. Now, since $M_{\ula}^{|A|}$ and $|M_\ula^A|^\si$ are free as $\k$-modules 
and have the same rank, the result follows. 
\end{proof}

Let $y\in S^A(n,d)$.
By Lemmas~\ref{LSchurIso} and~\ref{L100216}, for any $\ula\in\La([1,n]\times I,d)$, we can write
\begin{equation}\label{EVAction}
y v_\ula=\sum_{\umu\in \La([1,n]\times I,d)} v_\umu h_{\umu,\ula}
\end{equation}
for some $h_{\umu,\ula}\in W_d^A$. 
If $\phi\colon S^A(n,d) \iso \End_{W_d^A} (M^A(n,d))$
is the isomorphism of Corollary~\ref{C100116}, then 
$$\phi(y)(m_\ula^A)=\sum_{\umu\in \La([1,n]\times I,d)} m_\umu^A h_{\umu,\ula}.
$$ 

Consider the right $W_d^{|A|}$-module 
\begin{equation}\label{EM(n,d)}
M^{|A|}(n,d):=\bigoplus_{\ula \in\La([1,n]\times I,d)} M_\ula^{|A|}.
\end{equation}
By Lemma~\ref{LDesup} and Proposition~\ref{PDesupM}, there exists $\psi(y)\in \End_{W_d^{|A|}}(M^{|A|}(n,d))$ such that for any $\ula\in \La([1,n]\times I,d)$, 
\begin{equation}\label{EPsi}
\psi(y)(m_\ula^{|A|})=\sum_{\umu\in \La([1,n]\times I,d)} m_\umu^{|A|} \si^{-1}(h_{\umu,\ula}).
\end{equation}

\begin{Corollary} \label{CImportant} 
The map $\psi \colon |S^A(n,d)|\to\End_{W_d^{|A|}}(M^{|A|}(n,d))$ is an algebra isomorphism. 
\end{Corollary}

\begin{Remark}\label{RGradings} 
{\rm 
If the superalgebra $A$ is graded, the theory of this section goes through, yielding gradings on $S^A(n,d)$ and $W_d^A$, as well as all the modules over them that were considered. 
To be more precise, $M_n(A)$ inherits a grading from $A$, and then so do $\Tens M_n(A)$ and $\Inv M_n(A)$. On the other hand, $W_d^A=A^{\otimes d}\otimes \k\Si_d$ is graded with $\k\Si_d$ in degree $0$. 
In particular, the isomorphism $\psi$ from Corollary~\ref{CImportant} is an isomorphism of graded algebras. 
}
\end{Remark}

\section{Schur doubles}\label{SSymm}

Let $d\in \Z_{>0}$ and $n\in \Z_{\ge 0}$.
For an $\O$-superalgebra $A$ which is free of finite rank as an $\O$-supermodule, we denote 
$$
{}'{D^A(n,d)}:={}'{D^d M_n (A)}\quad\text{and}\quad D^A(n,d):=D^d M_n (A).
$$
The main result of this section is Theorem~\ref{Tsym}, which roughly speaking asserts that $D^A(n,d)$ is a maximal symmetric subalgebra
 of ${}'{D^A(n,d)}$. But first, we develop the results of \S\ref{SSGen} on generation in this set-up. 

\subsection{Generating $D^A(n,d)$}\label{SSGeneratingND}
Let $T_A=A\oplus A^*$ be the trivial extension superalgebra of $A$, 
cf.~\S\ref{SSTEA}. 
In view of Lemma~\ref{LTMx}, we identify  $M_n(T_A)$ with $T_{M_n(A)}$ so that $\xi_{r,s}^{(a,\al)}\in M_n(T_A)$ corresponds to $(\xi_{r,s}^a, x_{s,r}^\al)\in T_{M_n(A)}$ 
for all $1\leq r,s\leq n$, $a\in A$ and $\al\in A^*$, where $x_{r,s}^\al\in M_n(A)^*$ is the element defined in (\ref{EXRSAl}).

We also identify ${}'{D^A(n,d)}={}'{D^d M_n(A)}$ with $\Inv^d T_{M_n(A)}$ via the explicit isomorphism of Theorem~\ref{TFund}. Combining this with the identification $T_{M_n(A)}=M_n(T_A)$ from the previous paragraph, we now can and do identify ${}'{D^A(n,d)}$ with $\Inv^d M_n(T_A)=S^{T_A}(n,d)$. Since $D^A(n,d)$ is a subsuperalgebra of ${}'{D^A(n,d)}$, we now identify it as a subsuperalgebra of $S^{T_A}(n,d)$. 
As $A_\0$ is a subalgebra of $T_A$, the algebra $S^{A_\0}(n,d)=\Inv^d M_n(A_\0)$ is also a subalgebra of the superalgebra $S^{T_A}(n,d)=\Inv^d M_n(T_A)$ in the natural way.

\begin{Theorem} \label{TGenerationDND} 
The subsuperalgebra $D^A(n,d)\subseteq S^{T_A}(n,d)$ is precisely the subalgebra generated by $S^{A_\0}(n,d)$ and the set
$
\{\xi_{1,1}^y*1^{\otimes (d-1)}\mid y\in T_A\}\subseteq S^{T_A}(n,d).
$
\end{Theorem}
\begin{proof}
This follows from Corollary~\ref{CGen}. Indeed, we consider  that corollary
 with $M_n(A)$ in place of $X$. Then, taking 
into account the identifications made in this subsection, $\Inv^d(A_\0)$ in Corollary~\ref{CGen} corresponds to $S^{A_\0}(n,d)$, and 
$Y$  in Corollary~\ref{CGen} corresponds to $M_n(A_\1\oplus A^*)$. It remains to take $U:=\{\xi_{1,1}^y\mid y\in A_\1\oplus A^*\}$, which is easily seen to satisfy the assumptions of Corollary~\ref{CGen}. 
\end{proof}

\begin{Corollary} \label{C130216} 
Let $A$ and $A'$ be $\O$-superalgebras which are free of finite rank as $\O$-supermodules. If we have an isomorphism $\phi\colon T_A\iso T_{A'}$ which restricts to an isomorphism $A_\0\iso A'_\0$, then the isomorphism ${}S^{T_A}(n,d)\iso S^{T_{A'}}(n,d)$ induced by $\phi$ restricts to an isomorphism $D^A(n,d)\iso D^{A'}(n,d)$.
\end{Corollary}

If $A_\1 A_\1=0$, then $M_n(A)_\1 M_n(A)_\1=0$, and so we have Turner's grading on $D^A(n,d)$, cf.~Lemma~\ref{LGrDNew}. 

\begin{Corollary} 
If $A_\1 A_\1=0$, then  $D^A(n,d)$ is generated by the elements of degrees $0$, $1$ and $2$ with respect to Turner's grading. 
\end{Corollary}

\begin{Corollary} \label{C110216_2} 
Let $n\geq d$. Then the subsuperalgebra $D^A(n,d)\subseteq S^{T_A}(n,d)$ is precisely the subalgebra generated by $S^{A_\0}(n,d)$ and the set
$$
\{\xi_{1,1}^y*E_{2,2}^{\otimes \la_2}*\dots* E_{n,n}^{\otimes \la_n} \mid y\in T_A,\ (\la_2,\dots,\la_n)\in\La(n-1,d-1)\}\subseteq S^{T_A}(n,d).
$$
\end{Corollary}
\begin{proof}
Let $D$ be the subalgebra generated by the elements in the statement of the corollary. Let $\la=(1,\la_2,\dots,\la_n)\in\La(n,d)$. 
Recalling the idempotent $\xi_\la=E_{1,1}*E_{2,2}^{\otimes \la_2}*\dots* E_{n,n}^{\otimes \la_n}$ from (\ref{EXiLa}) and using Lemma~\ref{LExercise}, we have 
$$\xi_\la(\xi_{1,1}^y*1^{\otimes (d-1)})\xi_\la=\xi_{1,1}^y*E_{2,2}^{\otimes \la_2}*\dots* E_{n,n}^{\otimes \la_n}.
$$
By Theorem~\ref{TGenerationDND}, this shows that $D\subseteq  D^A(n,d)$. For the reverse inclusion, it suffices to show that the elements of the form $\xi_{1,1}^y*1^{\otimes (d-1)}$ with 
$y\in T_A$ belong to $D$. 
For any $\la\in\La(n,d-1)$, define
$$
x(y,\la):=\xi_{1,1}^y*E_{1,1}^{\otimes \la_1}*E_{2,2}^{\otimes \la_2}*\dots* E_{n,n}^{\otimes \la_n}\in S^{T_A}(n,d). 
$$
Then $\xi_{1,1}^y*1^{\otimes (d-1)}=\sum_{\la\in\La(n,d-1)} x(y,\la)$, so
 it suffices to prove that each $x(y,\la)\in D$. Fix $\la\in \La(n,d-1)$. Since $d-1<n$, there is $k\in[1,n]$ with $\la_k=0$. Let
\begin{align*}
b&:=E_{1,1}*E_{1,2}^{\otimes \la_1}*\dots*E_{k-1,k}^{\otimes \la_{k-1}}*E_{k+1,k+1}^{\otimes \la_{k+1}}*\dots*E_{n,n}^{\otimes \la_{n}},
\\
b'&:=E_{1,1}*E_{2,1}^{\otimes \la_1}*\dots*E_{k,k-1}^{\otimes \la_{k-1}}*E_{k+1,k+1}^{\otimes \la_{k+1}}*\dots*E_{n,n}^{\otimes \la_{n}},
\\
c&:=\xi_{1,1}^y*E_{2,2}^{\otimes \la_1}*\dots*E_{k,k}^{\otimes \la_{k-1}}*E_{k+1,k+1}^{\otimes \la_{k+1}}*\dots* E_{n,n}^{\otimes \la_n}.
\end{align*}
Then $b,b',c\in D$, and 
$
bcb'=x(y,\la)
$ 
by  Lemma~\ref{LExercise}, completing the proof. 
\end{proof}

For every $\la\in\La(n,d)$, the idempotent $\xi_\la\in S^{T_A}(n,d)$ defined in (\ref{EXiLa}) belongs to $S^{A_\0}(n,d)$, and so, by Corollary~\ref{C110216_2}, to $D^A(n,d)\subseteq S^{T_A}(n,d)$. 
The following is known, cf.~\cite[Lemma 13]{TurnerT}:

\begin{Corollary} 
If $d\leq n$, then $\xi_\om D^A(n,d)\xi_\om=\xi_\om S^{T_A}(n,d)\xi_\om$ and there is a superalgebra isomorphism 
$$
\phi\colon W^{T_A}_d\iso \xi_\om D^A(n,d)\xi_\om,\ (x_1\otimes\dots\otimes x_d)g\mapsto \xi^{x_1}_{1,g^{-1}1}*\dots*\xi^{x_d}_{d,g^{-1}d}.
$$ 
\end{Corollary}
\begin{proof}
First, we claim that every element of the form $\xi(x_1,\dots,x_d;g):=\xi^{x_1}_{1,g^{-1}1}*\dots*\xi^{x_d}_{d,g^{-1}d}$ belongs to $D^A(n,d)$. Indeed, the case $g=1$ is handled using Lemma~\ref{LExercise} and Corollary~\ref{C110216_2}, and the case $x_1=\dots=x_d=1$ is clear since $\xi(1,\dots,1;g)\in S^{A_\0}(n,d)$. 
By Lemma~\ref{LIdempotentTruncation}, the elements $\xi(x_1,\dots,x_d;g)$ span in $S^{T_A}(n,d)$ a copy of $W_d^{T_A}$, with the elements 
$\xi(x_1,\dots,x_d;1)$ spanning $T_A^{\otimes d}$ and the elements $\xi(1,\dots,1;g)$ spanning $\k\Si_d$. The claim follows.
 
Now, using Lemma~\ref{LIdempotentTruncation} and Corollary~\ref{C110216_2}, we conclude that $\xi_\om D^A(n,d)\xi_\om=\xi_\om S^{T_A}(n,d)\xi_\om$, and another application of  Lemma~\ref{LIdempotentTruncation} completes the proof.
\end{proof}

\subsection{Symmetric lattices in ${}'{D^A(n,d)}$}
\label{SSSymLat}

In this subsection, in addition to the hypotheses specified at the beginning of Section~\ref{Sprel}, we assume that $\O$ is a {\em principal ideal domain}. 
Let $A$ be an $\O$-superalgebra which is free of finite rank as an 
$\O$-supermodule. 
The following result shows, in particular, that $D^A(n,d)$ is maximal among the symmetric subalgebras of\, ${}'{D^A(n,d)}$.
The superstructure on  
${}'{D^A (n,d)}$ plays no role in the theorem, so the content of the statement does not change if ${}'D^A (n,d)$ is replaced by  $|{}'D^A (n,d)|$ and $D^A(n,d)$ by $|D^A (n,d)|$.

\begin{Theorem}\label{Tsym}
Let $d\in \Z_{\ge 0}$, $n\in\Z_{>0}$, and assume that $d\le n$.
Let $C$ be an $\O$-subalgebra of\, ${}'{D^A(n,d)}$ such that 
$D^A(n,d)\subseteq C \subseteq {}'{D^A(n,d)}$. 
Suppose that for every maximal ideal $\m$ of $\O$ 
the $(\O/\m)$-algebra 
$C\otimes_{\O} (\O/\m)$ is symmetric. Then $C=D^A(n,d)$. 
\end{Theorem}

\begin{proof}
If the theorem is true in the case where $\O$ is a discrete valuation ring (DVR), then it is true in general. 
Indeed, for every maximal ideal $\m$ of $\O$ the localisation 
$\O_{\m}$ is a DVR, and so by the DVR case of the theorem,
$C\otimes_{\O} \O_{\m}$ is equal to the $\O_\m$-span of 
$D^A(n,d) \otimes 1_{\O_\m}$. 
Then we have $(C / D^A(n,d)) \otimes_{\O} \O_\m=0$ for all $\m$, 
whence the $\O$-module $C/ D^A(n,d)$ is $0$.

In the rest of the proof, $\O$ is a DVR with the maximal ideal $(\pi)$ for some $\pi\in\O$, and $\k:=\O/(\pi)$. 
For any free $\O$-module $V$ of finite rank, we have the $\k$-vector space $V_\k=V\otimes_{\O} \k$, which we identify with  $V/\pi V$. 

Recall the notation (\ref{EDEF}) and (\ref{EDEFPrime}). 
In this proof, for all $e=0,\dots,d$, we use the following shorthands:
\begin{align*}
D&:=D^A(n,d), \quad {}'{D}:={}'{D^A(n,d)},\quad S:=S^A(n,d), 
\\
D^{d-e,e}&:= D^{d-e,e}M_n(A)=\Inv^{d-e} M_n (A) \otimes \Sym^e (M_n(A)^*)\subseteq D,
\\ 
{}'{D^{d-e,e}}&:= {}'{D^{d-e,e}M_n(A)}=\Inv^{d-e} M_n (A) \otimes {}'{\Sym^e (M_n(A)^*})\subseteq {}'{D},
\\
{}'{D^{>e}}&:=\bigoplus_{e+1\leq f\leq d} {}'{D^{d-f,f}}, \quad C^{d-e,e}: = C\cap {}'{D^{d-e,e}},\quad C^{>e}:=C\cap {}'{D^{>e}}. 
\end{align*} 

\vspace{2mm}
\noindent
{\em Claim 1.} The $\O$-submodule $C^{d-e,e}$ is pure in $C$.

\vspace{1mm}
\noindent
This follows immediately from the fact that ${}'{D^{d-e,e}}$ is pure in ${}'{D}$.

\vspace{2mm}
\noindent
{\em Claim 2.} We have $C^{d,0} = D^{d,0}={}'{D^{d,0}}$ and $C=C^{d,0}\oplus C^{>0}$.

\vspace{2 mm}
\noindent
Since ${}'{D^{d,0}}=D^{d,0}$ by definition, the assumption $D\subseteq C\subseteq {}'{D}$ implies that $C^{d,0} = D^{d,0}$. The second assertion of the claim follows easily from the first one.

\vspace{2mm}

\noindent
{\em Claim 3.} We have $\dim C^{d,0}_\k=\dim C^{0,d}_\k$. 

\vspace{1mm}
\noindent
Indeed, 
$$
\dim C^{d,0}_\k=\rank_\O C^{d,0}=\rank_\O D^{d,0}=\rank_\O D^{0,d}=\rank_\O C^{0,d}=\dim C^{0,d}_\k,
$$
where the penultimate equality comes from  $D^{0,d} \subseteq  C^{0,d} \subseteq {}'{D^{0,d}}$. 

\vspace{2 mm}

Since ${}'{D^{0,d}}$ is an ideal in ${}'{D}$ and $C\subseteq {}'{D}$ is a subalgebra, $C^{0,d}$ is an ideal in $C$, and so 
naturally a $C^{d,0}$-bimodule. After extending scalars, $C^{0,d}_\k$ becomes a $C^{d,0}_\k$-bimodule.

\vspace{2mm}
\noindent
{\em Claim 4.} The $C^{d,0}_{\k}$-bimodule $C^{0,d}_\k$ is isomorphic to $(C^{d,0}_{\k})^*$. 

\vspace{1mm}
\noindent
Since $C_\k$ is symmetric by assumption, there is a function $G\in C_\k^*$ such that the bilinear form on $C_\k$ defined by $(x,y):=G(xy)$ is symmetric and non-degenerate. By Claims 1 and 2, we can naturally identify  $C^{0,d}_\k$, $C^{d,0}_\k$, and $C^{>0}_\k$ with $\k$-subspaces of $C_\k$. Using the standard grading on ${}'{D}$, we see that the orthogonal complement to $C^{0,d}_\k$ in $C_\k$ contains  $C^{>0}_\k$. Comparing dimensions using Claim 3, we deduce that $(\cdot,\cdot)$ restricts to a perfect pairing between $C^{0,d}_\k$ and  $C^{d,0}_\k$, which yields the required isomorphism.

\vspace{2 mm}
In view of Remark~\ref{R123}, we  identify $D^{d,0}$ with $S$ and $D^{0,d}$ with $S^*$, so that $S^*\subseteq C^{0,d}\subseteq {}'{D^{0,d}}$. Extending scalars to the field of fractions $\K$ of $\O$, we identify $S^*_\K= C^{0,d}_\K= {}'{D^{0,d}_\K}$, and so we can consider $C^{0,d}$ and ${}'{D^{0,d}}$ as $\O$-submodules of $S^*_\K$. Now, define 
\begin{align*}
{}'{S} = \{ x\in S_{\K} \mid  
\langle x, {}'{D^{0,d}}  \rangle \subseteq \O\}
\quad\text{and}\quad 
N = \{ x\in S_\K  \mid \langle x, C^{0,d} \rangle \subseteq \O\}.
\end{align*}
Then ${}'{S} \subseteq N \subseteq S$. The following claim follows easily from the definitions:

\vspace{2mm}
\noindent
{\em Claim 5.}  
We have that ${}'{S}$ and $N$ are $S$-subbimodules of $S$, and there are isomorphisms of $S$-bimodules  ${}'{S}\cong ({}'{D^{0,d}})^*$ and $N\cong (C^{0,d})^*$. 
\vspace{2mm}

\noindent
{\em Claim 6.}  
We have $\xi_\om\in {}'{S}$.

\vspace{1mm}
\noindent
By Corollary~\ref{CPerfPair}, we have ${}'{S}=\Star^d M_n(A)$, and the claim follows from the definition of $\xi_\om$. 

\vspace{2mm}
\noindent
{\em Claim 7.}  
We have $N= S$.

\vspace{1mm}
\noindent
By Claims 2, 4 and 5, we have  
isomorphisms $S_\k=C_\k^{d,0}\cong (C_\k^{0,d})^*\cong N_\k$ of $S_{\k}$-bimodules. Let $z+\pi N\in N/\pi N=N_\k$ be the image of $1\in S_\k$ under this isomorphism. Then $x(z+ \pi N) = (z+\pi N)x$ for all $x\in S_{\k}$. Since $\pi S\supseteq \pi N$, 
it follows that $z+ \pi S \in Z(S/\pi S) = Z(S_{\k})$. 
Since $S_{\k}$ is generated by $1$ as a left $S_{\k}$-module, 
$N_\k$ is generated by $z + \pi N$ as a left $S_\k$-module. 
Moreover, $\xi_{\om} \in {}'{S}$ by Claim 6, so $\xi_{\om} \in N$. 
Hence  there exists $y\in S_{\k}$ such that $y(z+\pi N) = \xi_{\om} + \pi N$, 
whence $y(z+ \pi S) = \xi_{\om} + \pi S$. By Lemma~\ref{LInvertible}, 
$z+ \pi S$ is invertible in 
$S_{\k}$. So $N+\pi S = S_{\k} (z+ \pi S) S_{\k} = S_{\k}$. By Nakayama's Lemma, this implies that $N=S$. 

\vspace{2mm}
Now we complete the proof of the theorem. By Claim 7, we have $C^{0,d} = D^{0,d}$. 
Assume for a contradiction that $C\ne D$. 
Choose an element $x\in C \setminus D$ such that 
$x$ lies in ${}'{D^{>e-1}}$ with $e$ maximal possible. Then we can write $x= x_e  + \cdots + x_d$, where  
$x_f\in {}'{D^{d-f,f}}$ for $f=e,\dots,d$. 
By the maximality of $e$, we have 
$x_e \notin D^{d-e,e}$. Hence $x_e=cy$ for some $c\in\K\setminus \O$ and $y\in D^{d-e,e} \setminus \pi D^{d-e,e}$.

Let $F \in (D_\k)^*$ be as in Lemma~\ref{LFunction}. Taking into account Corollary~\ref{CSymmetric} and the standard grading on $D$, we conclude that there exists $u\in D^{e,d-e}$ such that 
$F(yu+\pi D) \ne 0$ in $\k$, whence
$y u \notin \pi D^{0,d}$. By the standard grading again,  $x_{f} u =0$ for all $f>e$, and hence $x u= c y u$. Since $c\notin \O$, it follows that $xu \notin D^{0,d} = C^{0,d}$. This is a contradiction, since $x\in C$ and $u\in D^{e,d-e}\subseteq C$. 
\end{proof}

\begin{Example} 
{\rm 
Continuing with Example~\ref{ExTriv}, assume that $d=2e+1$ for some $e\in\Z_{>0}$, and define the $\Z$-algebra $C$ to be the subalgebra of $\Q[z]_d$ spanned over $\Z$ by the elements $1,z,\dots,z^e,z^{e+1}/2,\dots,z^{2e+1}/2$. 
We then have $D^d \Z\cong\Z[z]_d\subsetneq C\subsetneq {}'{\Z[z]_d}\cong {}'{D^d \Z}$. However, it is easy to see that $C\otimes_\Z \F_p$ is symmetric for all primes $p$. This shows that the assumption $d\leq n$ in Theorem~\ref{Tsym} is essential. 
}
\end{Example}

\subsection{Bases and product rules}\label{SSBases}
Let $\ttB_\0$ be an $\O$-basis of $A_\0$, $\ttB_\1$ be an $\O$-basis of $A_\1$, and $\ttB=\ttB_\0\sqcup \ttB_\1$. 
The {\em structure constants} $\kappa^\ttb_{\ttb'\ttb''}\in\O$ of $A$ are determined from 
$$\ttb'\ttb''=\sum_{\ttb\in \ttB}\kappa^\ttb_{\ttb'\ttb''}\ttb
\qquad(\ttb,\ttb'\in \ttB). 
$$
Then 
\begin{equation}\label{EBasisM_n(A)}
\{\xi_{r,s}^\ttb\mid 1\leq r,s\leq n,\ \ttb\in \ttB\}
\end{equation}
 is a homogeneous basis of $M_n(A)$ with $\bar \xi_{r,s}^\ttb=\bar \ttb$, and 
\begin{equation}\label{EXiProduct}
\xi_{r,s}^{\ttb'}\xi_{t,u}^{\ttb''}=\de_{s,t}\sum_{\ttb\in \ttB}\kappa^\ttb_{\ttb'\ttb''}\xi_{r,u}^{\ttb}
\qquad(\ttb',\ttb''\in \ttB,\ 1\leq r,s,t,u\leq n).
\end{equation}
We fix a total order $<$ on the basis (\ref{EBasisM_n(A)}) as follows. First, we fix a total order $<$ on 
$\ttB$ so that the elements of $\ttB_\0$ precede the elements of $\ttB_\1$. Then for $\ttb',\ttb''\in \ttB$ and $1\leq r,s,t,u\leq n$, we set $\xi_{r,s}^{\ttb'}< \xi_{t,u}^{\ttb''}$ if and only if one of the following happens: (1) $\ttb'<\ttb''$, (2) $\ttb'=\ttb''$ and $r<t$, (3) $\ttb'=\ttb''$, $r=t$ and $s<u$. 

Recall the notation of~\S\ref{SSPairsSeq}. 
For $\bC=(C^\ttb)_{\ttb\in \ttB}\in\Mat^\ttB(n)$, we have the element 
$$
\xi_\bC:=*_{\ttb,r,s}\, \big( (\xi^\ttb_{r,s})^{\otimes c^\ttb_{r,s}} \big) \in S^A(n),
$$ 
where the $*$-product is taken in the order just defined. This agrees with (\ref{EBasis3}), so 
\begin{equation*}\label{EBasesSP}
\{\xi_\bC\mid\bC\in \Mat^\ttB(n)\}\quad\text{and}\quad
\{\xi_\bC\mid\bC\in \Mat^\ttB(n,d)\}
\end{equation*}
are bases of $S^A(n)$ and $S^A(n,d)$, respectively. The parity of a basis element is 
$
\bar \xi_\bC=\bar \bC:=|\bC|_\1\pmod 2 \!.
$ 

Let $\bC = (C^\ttb)_{\ttb\in \ttB} \in \Mat^\ttB (n,d)$ and 
$(\br,\ttbb,\bs) \in \bC$. Let $(\br^0,\ttbb^0,\bs^0)\in \bC$ be the tuple defined by the property that the triples 
$(r^0_1, \ttb^0_1, s^0_1),\dots, (r^0_d,\ttb^0_d, s^0_d)$ appear in the increasing order, i.e.~for $1\le k\le l\le d$ we have 
$\xi_{r^0_k,s^0_k}^{\ttb^0_k} \le \xi_{r^0_l,s^0_l}^{\ttb^0_l}$.
Let $g\in \Si_d$ be an element such that 
$(\br^0, \ttbb^0, \bs^0)g = (\br,\ttbb, \bs)$, 
and define 
\[
 [\br, \ttbb, \bs]:= [g; \ttb^0_1, \dots,\ttb^0_d ],
\]
cf.~\eqref{EGSign}. It follows from the definition of $\Seq^\ttB(n,d)^2$ that $[\br, \ttbb, \bs]$ does not depend on the choice of $g$. 
By the definition of the $*$-product, we have 
\begin{equation}\label{EXiC}
\xi_{\bC}  = \sum_{(\br,\ttbb,\bs)\in \bC} 
(-1)^{[\br, \ttbb, \bs]}\,
\xi_{r_1,s_1}^{\ttb_1} \otimes \cdots \otimes \xi_{r_d,s_d}^{\ttb_d}.
\end{equation}

For $\bC,\bD\in \Mat^\ttB(n)$, we let  
\begin{equation*}\label{EEpsBABB}
\eps_{\bC\bD}:=
\left\{
\begin{array}{ll}
(-1)^{\sum c^{\ttb'}_{r,s}d^{\ttb''}_{t,u}} &\hbox{if $\bC+\bD\in\Mat^\ttB(n)$,}\\
0 &\hbox{otherwise.}
\end{array}
\right.
\end{equation*}
where the summation is over all $1\leq r,s,t,u\leq n$ and $\ttb',\ttb''\in \ttB_\1$ such that $\xi_{r,s}^{\ttb'}>\xi_{t,u}^{\ttb''}$. Using  Lemma~\ref{LstDe}, we obtain for all $\bC\in \Mat^\ttB(n)$:
\begin{equation} \label{ECopOnSi} 
\De(\xi_\bC)=\sum_{\bD,\bE\in\Mat^\ttB(n),\ \bD+\bE=\bC}\eps_{\bD\bE}\,\xi_{\bD}\otimes\xi_{\bE}.
\end{equation}

Define the structure constants $f_{\bC\bD}^\bE\in\O$ from
\begin{equation}\label{EFBABBBCNew}
\xi_\bC\xi_\bD=\sum_{\bE\in\Mat^\ttB(n)} f_{\bC\bD}^\bE \xi_\bE\qquad(\bC,\bD\in\Mat^\ttB(n)). 
\end{equation}
In particular, $f_{\bC\bD}^\bE =0$ unless $|\bC|=|\bD|=|\bE|$. 
These structure constants are uniquely determined by the structure constants $\kappa_{\ttb',\ttb''}^\ttb$. More precisely, if 
$(\br,\ttbb, \bs) \in \Seq^{\ttB} (n,d)^2$ and $\bE=\bM[\br,\ttbb,\bs]$, then using~\eqref{EXiProduct} and~\eqref{EXiC} we obtain the formula
\begin{equation*}
f_{\bC\bD}^{\bE} = \sum_{\ttbb', \bt, \ttbb''} 
(-1)^{[\br, \ttbb, \bs]+ [\br, \ttbb', \bt]+[\bt,\ttbb'', \bs]+[\ttb'_1,\dots,\ttb'_d; \ttb''_1,\dots,\ttb''_d]} \, 
\kappa_{\ttb'_1,\ttb''_1}^{\ttb_1}  \cdots 
\kappa_{\ttb'_d,\ttb''_d}^{\ttb_d}
\end{equation*}
where the sum is over all triples
$(\ttbb',\bt, \ttbb'') \in \ttB^d \times \Seq(n,d) \times \ttB^d$
such that $(\br, \ttbb', \bt) \in \bC$ and $(\bt, \ttbb'', \bs)\in \bD$. 
In the case when $A=\O$ and $\ttB=\{1 \}$, this is 
Green's formula \cite[(2.3b)]{Green} for the 
structure constants of the Schur algebra.

Let  
$
\{x^\bC\mid \bC\in\Mat^\ttB(n)\}
$
be the basis of $S^A(n)^*=(\Inv M_n(A))^*$ dual to the basis $\{\xi_\bC\mid\bC\in \Mat^\ttB(n)\}$ of $S^A(n)$. As the product and the coproduct on $S^A(n)^*$ are by definition dual to the coproduct and the product on $S^A(n)$, respectively, we have in view 
of (\ref{ECopOnSi}), (\ref{EFBABBBCNew}) and~\eqref{EDualTensor}:
\begin{eqnarray}\label{EProduct2}
x^{\bC}x^\bD&=&
(-1)^{\bar\bC\bar\bD}\eps_{\bC\bD}\, x^{\bC+\bD}\qquad(\bC,\bD\in \Mat^\ttB(n)),
\\
\label{ECoProduct2}
\nabla(x^\bC)&=&\sum_{\bD,\bE\in\Mat^\ttB(n,d)}(-1)^{\bar\bD\bar\bE}f^\bC_{\bD\bE} x^\bD\otimes x^\bE\qquad(\bC\in\Mat^\ttB(n,d)).
\end{eqnarray}
It is easy to see that $S^A(n)^*$ is the free supercommutative superalgebra on the even variables $\{x_{r,s}^\ttb\mid \ttb\in \ttB_\0,\ 1\leq r,s\leq n\}$ and the odd variables $\{x_{r,s}^\ttb\mid \ttb\in \ttB_\1,\ 1\leq r,s\leq n\}$, and 
$$
x^{\bC}=(-1)^{|\bC|_\1(|\bC|_\1-1)/2}\prod_{\ttb\in \ttB,\ 1\leq r,s\leq n}(x_{r,s}^\ttb)^{c_{r,s}^\ttb},
$$
with the product taken in the total order on the variables $x_{r,s}^\ttb$ which is the same as the one on the basis $\{\xi_{r,s}^\ttb\}$ fixed above. 

Let $x^{(\bC)}:=\frac{x^\bC}{\bC!}$. 
By (\ref{EProduct2}), we have  
\begin{equation*}\label{EDivPowBaNew}
x^{(\bC)}x^{(\bD)}=
(-1)^{\bar\bC\bar\bD}\eps_{\bC\bD}\, {\bC+\bD\choose \bD}\, x^{(\bC+\bD)}
\qquad(\bC,\bD\in\Mat^\ttB(n)).
\end{equation*}
Then ${}'{\Sym(M_n(A)^*)}$ is the $\O$-span in $S^A(n)^*_\K$ of all $x^{(\bC)}$ with $\bC\in \Mat^\ttB(n)$. 
Let 
$$
f^{(\bE)}_{(\bC)\bD}:=\frac{f_{\bC\bD}^\bE \bC! }{\bE!}
\quad
\text{and}
\quad 
f^{(\bE)}_{\bC(\bD)}:=\frac{ f_{\bC\bD}^\bE \bD!}{\bE!}\qquad (\bC,\bD,\bE\in\Mat^\ttB(n)).
$$
A priori, these are elements of $\K$, but by Lemma~\ref{LDeinv}, they actually belong to $\O$, and for $\bC\in\Mat^\ttB(n,d)$ we have 
\begin{equation*}\label{ECoprodIntSmallNew}
\nabla(x^{(\bC)})=\sum_{\bD,\bE\in \Mat^\ttB(n,d)}\hspace{-2mm}(-1)^{\bar\bD\bar\bE}f^{(\bC)}_{(\bD)\bE}\, x^{(\bD)}\otimes x^{\bE}
=\sum_{\bD,\bE\in \Mat^\ttB(n,d)}\hspace{-2mm}(-1)^{\bar\bD\bar\bE}f^{(\bC)}_{\bD(\bE)}\, x^{\bD}\otimes x^{(\bE)}.
\end{equation*}

Denoting 
\begin{equation*}\label{EOP2}
\Mat^\ttB_2(n,d):=\{(\bC,\bD)\mid \bC,\bD\in\Mat^\ttB(n),\ |\bC|+|\bD|=d\},
\end{equation*}
we have bases
$$
\{\xi_\bC\otimes x^\bD\mid (\bC,\bD)\in\Mat^\ttB_2(n,d)\}\quad 
\text{and}\quad
\{\xi_\bC\otimes x^{(\bD)}\mid (\bC,\bD)\in\Mat^\ttB_2(n,d)\}
$$
of $D^A(n,d)$ and $'{D^A(n,d)}$, respectively. If $A_\1 A_\1=0$, then $M_n(A)_\1 M_n(A)_\1=0$, and the Turner gradings on $D^A(n,d)$ and ${}'{D^A(n,d)}$ satisfy 
$$
\deg(\xi_\bC\otimes x^\bD)=\deg(\xi_\bC\otimes x^{(\bD)})=|\bC|_\1+2|\bD|_\0+|\bD|_\1,
$$
for all $(\bC,\bD)\in\Mat^\ttB_2(n,d)$, cf. Lemma~\ref{LGrDNew}.

For $(\bC,\bD),(\bE,\bF)\in\Mat^\ttB_2(n,d)$, we have the following {\em product rules}, which come from (\ref{EProductDoubleGen}):
\begin{equation*}
\begin{split}
(\xi_\bC\otimes x^\bD)(\xi_\bE\otimes x^\bF)
&=\sum (-1)^s
\eps_{\bC_1\bC_2}
\eps_{\bE_1\bE_2}
 f^\bD_{\bE_2 \bD'}
f^\bF_{\bF' \bC_1}
(\xi_{\bC_2}\xi_{\bE_1}\otimes x^{\bD'}x^{\bF'})
\\
&=\sum (-1)^t
\eps_{\bC_1\bC_2}
\eps_{\bE_1\bE_2}
\eps_{\bD' \bF'}
f^\bD_{\bE_2 \bD'}
f^\bF_{\bF' \bC_1}
f^\bG_{\bC_2 \bE_1}
(\xi_{\bG}\otimes x^{\bD'+\bF'}),
\end{split}
\end{equation*}
\begin{equation*} 
\begin{split}
(\xi_\bC\otimes x^{(\bD)})(\xi_\bE\otimes x^{(\bF)})
&=\sum (-1)^s
\eps_{\bC_1\bC_2}
\eps_{\bE_1\bE_2}
f^{(\bD)}_{\bE_2 (\bD')}
f^{(\bF)}_{(\bF') \bC_1}
(\xi_{\bC_2}\xi_{\bE_1}\otimes x^{(\bD')}x^{(\bF')})
\\
&=\sum (-1)^t
\eps_{\bC_1\bC_2}
\eps_{\bE_1\bE_2}
\eps_{\bD' \bF'}
f^{(\bD)}_{\bE_2 (\bD')}
f^{(\bF)}_{(\bF') \bC_1}
\\
&
\quad\quad\quad
\times
f^\bG_{\bC_2 \bE_1}
{\bD'+\bF'\choose \bD'}
 (\xi_{\bG}\otimes x^{(\bD'+\bF')})
\end{split}
\end{equation*}
where the first sums in both formulas are over all $\bC_1,\bC_2,\bD',\bE_1,\bE_2,\bF'\in\Mat^\ttB(n)$ such that $\bC_1+\bC_2=\bC$, $\bE_1+\bE_2=\bE$, the second sums have an additional summation parameter $\bG\in\Mat^\ttB(n)$, and
\begin{eqnarray*}
s=\bar\bC_1\bar\bC_2+\bar\bC_1\bar\bE_1+\bar \bC_1\bar \bD'+ \bar \bD'\bar \bE_1+\bar \bE_1 \bar \bE_2,\quad 
t=s+\bar \bD'\bar \bF'.
\end{eqnarray*}

\section{The quiver case}\label{SQuiver}

In this section we consider an important class of algebras $D^A(n,d)$ sometimes referred to as {\em schiver doubles}, from  `schiver=Schur+quiver' \cite{Turner}. 

\subsection{Quivers and quiver algebras}\label{SSQ}
Let $Q$ be a quiver with a finite set of vertices $I=\{1,\dots,l\}$ and a finite set of directed edges $E$. For an edge $\be\in E$, we denote by $s(\be)\in I$ the source of $\be$ and by $t(\be)\in I$ the target of $\be$. 
We denote by $\Gamma$ the underlying graph of $Q$. 
We assume that $\Gamma$ is connected and has no loops or multiple edges. 
If $i,j\in I$, we say that $i$ and $j$ are {\em neighbors} if they are connected by an edge in $\Gamma$. 

We define the algebra $P_Q$ to be the quotient of the path algebra $\k Q$ by all quadratic relations. We consider $P_Q$ as a  superalgebra with vertices in parity $\0$ and edges in parity $\1$. The parity $\0$ component $P_{Q,\0}$ has a basis $\{e_i\mid i\in I\}$, and the parity $\1$ component $P_{Q,\1}$ has  a basis $\{\be\mid \be\in E\}$. Note that  $P_{Q,\1}P_{Q,\1}=0$, so $P_Q$ is also $\Z$-graded with the degree $0$ component $P_{Q}^0=P_{Q,\0}$ and 
degree $1$ component $P_{Q}^1=P_{Q,\1}$.

Let $\{e_i^*, \be^* \mid i\in I,\be\in E\}$ be the basis of $P_Q^*$ dual to the basis $\{e_i,\be\mid i\in I,\be\in E\}$ of $P_Q$. 
According to the agreement made in \S\ref{SSGr}, we always work with the $\Z$-grading on $P_Q^*$ which is the {\em shift by $2$} of the canonical grading, i.e.~$\deg e_i^*=2$ and $\deg \be^*=1$ for all $i\in I$ and $\be\in E$. Then the trivial extension superalgebra  $T_{P_Q}=P_Q\oplus P_Q^*$ is also graded. This superalgebra has an easy description as a zigzag algebra, which we introduce next.

The {\em  zigzag algebra $\Zig=\Zig_\Ga$ of type $\Gamma$} is defined in \cite{HK} as follows. 
First assume that $l>1$. Let $\overline{\Gamma}$ be the quiver obtained by doubling all edges in $\Ga$ and then orienting the edges so that if $i$ and $j$ are neighboring vertices in $\Gamma$, then there is a directed edge $\za^{i,j}$ from $j$ to $i$ and a directed edge $\za^{j,i}$ from $i$ to $j$. 
Then $\Zig$ is the path algebra $\k\overline{\Gamma}$, generated by length $0$ paths $\ze_i$ for $i\in I$ and length $1$ paths $\za^{i,j}$, subject only to the following relations:
\begin{enumerate}
\item All paths of length three or greater are zero.
\item All paths of length two that are not cycles are zero.
\item All length-two cycles based at the same vertex are equal.
\end{enumerate}
The algebra $\Zig$ inherits the path length grading from $\k\overline{\Gamma}$. 
If $l =1$, i.e.~$\Gamma$ is of type ${A}_1$, we define $\Zig_{A_1}:= \k[\zc]/(\zc^2)$, where $\zc$ is an indeterminate in degree $2$. 

If $l>1$, for every vertex $i$ pick its neighbor $j$ and denote $\zc^{(i)}:=\za^{i,j}\za^{j,i}$.  The relations in $\Zig$ imply that 
$\zc^{(i)}=\ze_i\zc^{(i)} \ze_i$ 
is independent of choice of $j$. Define $\zc:= \sum_{i \in V} \zc^{(i)}$. 
 Then in all cases $\Zig$ has a basis
$$
\{\za^{i,j} \mid i,j \in I,\ j \textup{ is a neighbor of }i\} \cup \{\zc^m\ze_i \mid i \in I,\ m \in \{0,1\}\}, 
$$
and the graded $\k$-rank of $\Zig$ equals 
$
l(1+q^2)+2 |E|q\in\Z[q],
$ where $q$ is an indeterminate. Moreover, we consider $\Zig$ as a superalgebra with $\Zig_{\0}=\Zig^0\oplus\Zig^2$ and $\Zig_{\1}=\Zig^1$. 

The following is known \cite[Lemma 6]{TurnerT} and easy to check:

\begin{Lemma} \label{LTZ} 
There is an isomorphism of graded superalgebras $T_{P_Q}\iso \Zig$
given by 
$
e_i \mapsto \ze_i, \; e_i^* \mapsto \zc^{(i)}, \; \be\mapsto \za^{i,j}, \; 
\be^* \mapsto \za^{j,i}
$
if $s(\be)=j$ and $t(\be)=i$. 
\end{Lemma}

\subsection{Schiver doubles} \label{SSSchiver}
From now on we will work over $\O$. 
For a quiver $Q$  as in the previous subsection, we define
$$
D_Q(n,d):=D^{P_Q}(n,d),\quad {}'{D_Q(n,d)}:={}'{D^{P_Q}(n,d)}.
$$
In view of Lemma~\ref{LTZ}, we identify $T_{P_Q}$ with $\Zig$, and so, as in \S\ref{SSGeneratingND}, we identify ${}'{D_Q(n,d)}$ with $S^{T_{P_Q}}(n,d)=S^{\Zig}(n,d)$. In this way, we identify $D_Q(n,d)$ with a subalgebra of $S^{\Zig}(n,d)$. 
By Corollary~\ref{C130216}, the superalgebra $D_Q(n,d)$ does not depend on the choice of orientation on $Q$, cf.~\cite[Theorem 157]{Turner}. 
As $P_{Q,\1}P_{Q,\1}=0$, we have Turner's gradings on $D_Q(n,d)$ and ${}'{D_Q(n,d)}$, see \S\ref{SSGr}. 
We also have a grading on $S^{\Zig}(n,d)$, see Remark~\ref{RGradings}.
All our identifications respect gradings. 

Note that the degree zero component of $P_Q$ is $P_{Q}^0=\sum_{i=1}^l\O e_i\cong \O^{\oplus l}$. Recall that $S(n,d)=S^\O(n,d)$ is the classical Schur algebra.   
By Corollary~\ref{CExercise}, 
\begin{equation}\label{E110216}
S^{P_{Q}^0}(n,d)\cong \bigoplus_{(d_1,\dots,d_l)\in\La(l,d)} S(n,d_1)\otimes\dots\otimes S(n,d_l).
\end{equation}

\begin{Lemma} \label{L110216_2} 
The image of the natural embedding $S^{P_{Q}^0}(n,d)\to S^{\Zig}(n,d)$ is exactly the degree zero component $S^{\Zig}(n,d)^0$. 
\end{Lemma}
\begin{proof}
As $P_{Q}^0=\Zig^{0}$, we have $M_n(P_{Q}^0)=M_n(\Zig)^0$, which implies the lemma. 
\end{proof}

\begin{Theorem} \label{TGenerationAgain} 
Let $n\geq d$. Then the subsuperalgebra $D_Q(n,d)\subseteq S^{\Zig}(n,d)$ is precisely the subalgebra generated by $S^{\Zig}(n,d)^0$ and the set
$$
\{\xi_{1,1}^z*E_{2,2}^{\otimes \la_2}*\dots* E_{n,n}^{\otimes \la_n} \mid z\in \Zig,\ (\la_2,\dots,\la_n)\in\La(n-1,d-1)\}\subseteq S^{\Zig}(n,d).
$$
\end{Theorem}
\begin{proof}
In view of Lemma~\ref{L110216_2}, this is a restatement of Corollary~\ref{C110216_2}. 
\end{proof}

Note that $D_Q(n,d)$, ${}'{D_Q(n,d)}$ and $S^{\Zig}(n,d)$ are graded {\em super}algebras, whose constructions depend on the superalgebra structures on $P_Q$ and $\Zig$. However, {\em after} we construct them, we want to forget the superalgebra structures and work with $D_Q(n,d)$ and $S^{\Zig}(n,d)$ as usual graded algebras. In order to do that, recall the theory of \S\ref{SSDesup}. From now on, we assume that $\Ga$ has no odd cycles. Then to every vertex $i\in I$, we can assign a sign $\zeta_i\in\{\pm 1\}$ such that $\zeta_i\zeta_j=-1$ whenever $i$ and $j$ are neighbors. 
Let 
$$e^\0=\sum_{i\in I,\ \zeta_i=1}\ze_i\quad \text{and}\quad e^\1=\sum_{i\in I,\ \zeta_i=-1}\ze_i.
$$ 
One can easily check that $(e^\0,e^\1)$ is an adapted pair of idempotents for the superalgebra $\Zig$ in the sense of \S\ref{SSDesup}.

By Lemma~\ref{LDesup}, there is an explicit isomorphism of graded algebras 
$\si\colon W_d^{|\Zig|}\iso |W_d^{\Zig}|$. Moreover, as in (\ref{EMLaC}) and (\ref{EM(n,d)}), we have the colored permutation modules $M_\ula^{|\Zig|}$ labeled by $\ula\in\La([1,n]\times I,d)$ and set
$$
M^{|\Zig|}(n,d):=\bigoplus_{\ula \in\La([1,n]\times I,d)} M_\ula^{|\Zig|}.
$$

For $\ula\in\La([1,n-1]\times I,d-1)$ and $k\in J$, we define 
$\hat\ula^{\!k}\in \La([1,n]\times I,d)$ by 
$$
\hat\ula^{\!k}_{(r,i)}=
\left\{
\begin{array}{ll}
\ula_{(r-1,i)} &\hbox{if $r>1$,}\\
1 &\hbox{if $r=1$ and $i= k$,}\\
0&\hbox{if $r=1$ and $i\neq k$.}
\end{array}
\right.
$$

\begin{Lemma}\label{LILAZ} 
Let $z\in \ze_j\Zig\ze_k$ for some $j,k\in I$ and $\ula\in\La([1,n-1]\times I,d-1)$. Then there is a unique 
$\mathtt{i}^\ula(z)\in \End_{W_d^{|\Zig|}}(M^{|\Zig|}(n,d))$ such that 
$$\mathtt{i}^\ula(z)\colon m_{\umu}^{|\Zig|}\mapsto 
\left\{
\begin{array}{ll}
(m_{\hat\ula^{\!j}}^{|\Zig|})z[1] &\hbox{if $\umu=\hat\ula^{\!k}$,}\\
0 &\hbox{otherwise,}
\end{array}
\right.
$$
where $z[1]=z\otimes 1_\Zig^{\otimes d-1}\in W_d^{|\Zig|}$. 
\end{Lemma}
\begin{proof}
Recalling~\eqref{EIdempTensor}, for any $\umu\in \La(n,d)$ set 
$e_{\umu}:=e_{\umu}^{|A|} \in \Tens^d |A| \subseteq W_d^{|A|}$.
Note that for all $i\in I$, we have 
$
e_{\hat\ula^{\!i}}=e_i\otimes e_\ula
$ and $\Si_{\hat\ula^{\!i}}=\Si_1\times \Si_\ula$. 
It follows that  
$$m_{\hat\ula^{\!j}}^{|\Zig|} \, z[1](e_{\hat\ula^{\!k}}\otimes g)=\eps_{\hat\ula^{\!k}}(g)m_{\hat\ula^{\!j}}^{|\Zig|} \, z[1]$$ for all $g\in\Si_{\hat\ula^{\!i}}$. 
By adjointness of induction and restriction, there exists a unique  map as in the statement. 
\end{proof}

Using the maps of Lemma~\ref{LILAZ}, define 
$$
\mathtt{i}^\ula\colon |\Zig|\to \End_{W_d^{|\Zig|}}(M^{|\Zig|}(n,d)), \ z\mapsto \sum_{j,k\in I}\mathtt{i}^\ula(\ze_j z\ze_k).
$$
The following is easy to see:

\begin{Lemma} 
For any $\ula\in\La([1,n-1]\times I,d-1)$, the map $\mathtt{i}^\ula$ is an injective homomorphism of graded algebras. 
\end{Lemma}

By Corollary~\ref{CImportant} and Remark~\ref{RGradings}, there is an explicit isomorphism of graded algebras
$$\psi\colon |S^{\Zig}(n,d)|\iso\End_{W_d^{|\Zig|}}(M^{|\Zig|}(n,d)).$$
We use this isomorphism to identify the graded algebra  $|{}'{D_Q(n,d)}|=|S^{\Zig}(n,d)|$ with the graded algebra $\End_{W_d^{|\Zig|}}(M^{|\Zig|}(n,d))$. 

\begin{Theorem}\label{TGenerationAgainAgain} Let $n\geq d$. The subalgebra $|D_Q(n,d)|\subseteq |S^{\Zig}(n,d)|$ is precisely the subalgebra generated by the degree zero component $|S^{\Zig}(n,d)|^0$ and the set 
$$\bigcup_{\ula\in\La([1,n-1]\times I,d-1)}\mathtt{i}^\ula(\Zig).$$ 
\end{Theorem}

\begin{proof}
Let $\la=(\la_1,\dots,\la_{n-1})\in\La(n-1,d-1)$ and $z\in e_j\Zig e_k$ for some $j,k\in I$. We claim that  
$$\xi_{1,1}^z*E_{2,2}^{\otimes \la_1}*\dots* E_{n,n}^{\otimes \la_{n-1}}
=\sum_{\ula\in\pi^{-1}(\la)}\mathtt{i}^\ula(z),
$$
which implies the result by Theorem~\ref{TGenerationAgain}. To prove the claim, let $\unu\in \La([1,n]\times I,d)$. 
Note that $(\xi_{1,1}^z*E_{2,2}^{\otimes \la_1}*\dots* E_{n,n}^{\otimes \la_{n-1}})v_\unu=0$ unless $\unu$ is of the form $\hat \umu^{k}$ for some $\umu\in\pi^{-1}(\la)$ and $k\in I$. Moreover, for $\umu\in\pi^{-1}(\la)$, we have 
\begin{align*}
(\xi_{1,1}^z*E_{2,2}^{\otimes \la_1}*\dots* E_{n,n}^{\otimes \la_{n-1}})v_{\hat \umu^{k}}&=zv_{1,k}\otimes v_\umu=v_{1,j}z\otimes v_\umu
\\
&=(v_{1,j}\otimes v_\umu)z[1]
=v_{\hat \umu^j}z[1],
\end{align*}
where $z[1]=z\otimes 1^{\otimes d-1}$ is viewed as an element of $W_d^\Zig$. Comparing with (\ref{EVAction}) and (\ref{EPsi}), we deduce that 
\begin{align*}
(\xi_{1,1}^z*E_{2,2}^{\otimes \la_1}*\dots* E_{n,n}^{\otimes \la_{n-1}})(m_{\hat \umu^k}^{|\Zig|})
&=
m_{\hat \umu^j}^{|\Zig|} \, \si^{-1}(z[1])
=m_{\hat \umu^j}^{|\Zig|}\, z[1]=
\mathtt{i}^\umu(z)(m_{\hat \umu^k}^{|\Zig|})
\\
&=\sum_{\ula\in\pi^{-1}(\la)}\mathtt{i}^\ula(z)(m_{\hat \umu^k}^{|\Zig|}),
\end{align*}
where we have used the fact that $\si^{-1}(z[1])=z[1]$, see Lemma~\ref{LDesup}.  The claim is proved. 
\end{proof}


\begin{thebibliography}{ABC}

\bibitem[DJ]{DJ}
R.\ Dipper and G.\ James, Representations of Hecke algebras and general linear groups, {\em Proc.\ London Math.\ Soc.\ (3)} {\bf 52}  (1986),  20--52.

\bibitem[EK]{EK2}
A.\ Evseev and A.\ Kleshchev, Blocks of symmetric groups, semicuspidal KLR algebras and zigzag  Schur-Weyl duality, {\em preprint}, 2016. 

\bibitem[Gr]{Green}
J.A.\ Green, {\em Polynomial representations of $GL_n$}, 2nd edition, Springer-Verlag, 2007.


\bibitem[HK]{HK} R.S.\ Huerfano and M.\ Khovanov, A category for the adjoint representation, {\em J.~Algebra} {\bf 246} (2001), 514--542.


\bibitem[JK]{JK} G.D.\ James and A.\ Kerber, {\em The Representation Theory of the Symmetric Groups}, Addison-Wesley, 1981. 

\bibitem[Re]{Re}
C.\ Reutenauer, {\em Free Lie algebras}, Oxford University Press, 1993.

\bibitem[Tu${}_1$]{Turner}
W.\ Turner, Rock blocks. Mem.\ Amer.\ Math.\ Soc.\ {\bf 202} (2009), no. 947. 

\bibitem[Tu${}_2$]{TurnerT}
W.\ Turner, Tilting equivalences: from hereditary algebras to symmetric groups, {\em J.\ Algebra} {\bf 319} (2008), 3975--4007. 

\bibitem[Tu${}_3$]{TurnerCat}
W.\ Turner, Bialgebras and caterpillars, {\em Q.\ J.\ Math.} {\bf 59}  (2008), 379--388.
\end{thebibliography}
\end{document}